\documentclass[11pt]{amsart}

\usepackage{amsaddr}
\usepackage{caption}
\usepackage{float}
\usepackage{graphicx}
\usepackage{geometry} 
\usepackage{setspace}
\usepackage{subcaption}
\usepackage{tabu}

\newcommand{\vx}{\vct{x}}

\newcommand{\vvu}{\vvct{u}}

\newcommand{\vvv}{\vvct{v}}

\newcommand{\valpha}{\boldsymbol\alpha}

\newcommand{\Pih}{\Pi_{h}}
\newcommand{\Lc}{\mathcal{L}}
\newcommand{\Lh}{\mathcal{L}_h}
\newcommand{\Lhs}{\mathcal{L}_{h*}}

\newcommand{\vct}[1]{\mathbf{#1}}
\newcommand{\vvct}[1]{\underline{#1}}
\newcommand{\ten}[1]{\mathbf{#1}}

\DeclareMathOperator*{\esssup}{ess\,sup}

\newtheorem{thm}{Theorem}[section]
\newtheorem{lem}[thm]{Lemma}
\newtheorem{cor}[thm]{Corollary}
\newtheorem{rem}[thm]{Remark}

\title[Pre and post-processing for the FEM for linear wave problems]{Doubling the convergence rate by pre- and post-processing the finite element approximation for linear wave problems}
\author{S. Geevers$^1$}
\address{1. Department of Mathematics, University of Vienna, Austria (e-mail: sjoerd.geevers@univie.ac.at)}
\thanks{*The author has been funded by the Austrian Science Fund (FWF) through the project F 65 `Taming Complexity in Partial Differential Systems}

\begin{document}

\maketitle

\begin{abstract}
In this paper, a novel pre- and post-processing algorithm is presented that can double the convergence rate of finite element approximations for linear wave problems. In particular, it is shown that a $q$-step pre- and post-processing algorithm can improve the convergence rate of the finite element approximation from order $p+1$ to order $p+1+q$ in the $L^2$-norm and from order $p$ to order $p+q$ in the energy norm, in both cases up to a maximum of order $2p$, with $p$ the polynomial degree of the finite element space. The $q$-step pre- and post-processing algorithms only need to be applied \emph{once} and require solving at most $q$ linear systems of equations. 

The biggest advantage of the proposed method compared to other post-processing methods is that it does not suffer from convergence rate loss when using unstructured meshes. Other advantages are that this new  pre- and post-processing method is straightforward to implement, incorporates boundary conditions naturally, and does not lose accuracy near boundaries or strong inhomogeneities in the domain. Numerical examples illustrate the improved accuracy and higher convergence rates when using this method. In particular, they confirm that $2p$-order convergence rates in the energy norm are obtained, even when using unstructured meshes or when solving problems involving heterogeneous materials and curved boundaries.
\end{abstract}

\section{Introduction}
When solving wave propagation problems, the finite element method offers a good alternative to the popular finite difference method when the effects of the geometry, e.g.  the geometry of objects in scattering problems or the topography in seismic models, need to be accurately modelled. The error of the finite element solution is at most of order $h^p$ in the energy norm and order $h^{p+1}$ in the $L^2$-norm, where $h$ denotes the mesh resolution and $p$ the polynomial degree of the finite element space. We show that, by carefully discretising the initial values and by post-processing only the final solution, we can improve the convergence rates in both norms up to order $2p$ . 

Post-processing methods for the finite element method have been known for several decades and have been applied to numerous problems, including elliptic problems \cite{bramble77,thomee77}, parabolic problems \cite{thomee80}, and first-order hyperbolic problems \cite{bales93,cockburn03}. For an historic overview, see, for example, \cite{cockburn03} and the references therein. These methods exploit the fact that the numerical error converges faster in a negative-order Sobolev norm than in the $L^2$-norm. Typically, the post-processed solution is obtained by convolving the finite element solution with a suitable kernel. The error of the post-processed solution can then be bounded by the numerical error and difference quotients of the numerical error in a negative-order Sobolev norm. 

In this paper, we show that the numerical error for the linear wave equation converges with rate $\min(p+1+m,2p)$ in the $(-m)$th-order Sobolev norm. Difference quotients of the numerical error converge with the same rate in case translation-invariant meshes are used. In case of unstructured meshes, however, the super-convergence rate of the difference quotients, and therefore the post-processed solution, is (partially) lost \cite{curtis07,mirzaee13,li16}.  We therefore present an alternative post-processing algorithm that fully preserves the super-convergence rate when using unstructured meshes. 

The accuracy of this new approach relies on (adapted) negative-order Sobolev norms of \emph{time derivatives} of the numerical error instead of difference quotients. These time derivatives also play an important role in the error analysis of finite element methods for nonlinear problems \cite{meng17}. They maintain the full super-convergence rate when using unstructured meshes if the initial values are carefully discretised. To improve the convergence rate by $q$ orders up to a maximum of order $2p$, the proposed algorithm requires solving at most $q$ elliptic problems for discretising the initial value and $q$ elliptic problems for post-processing the final solution. These elliptic problems are solved with a finite element method using the same mesh as for the time-stepping. For post-processing, elements of degree $p+q$ are used instead of degree $p$. The resulting linear systems of equations can be solved with a direct solver or with an iterative solver and for the latter, good initial guesses can be obtained from the unprocessed initial and final values.

The biggest advantage of this new method is that the super-convergence rates are fully maintained when using unstructured meshes. Other advantages are that the method is easy to implement, naturally incorporates the boundary conditions, and suffers no accuracy loss near boundaries or material interfaces, while post-processing with a convolution requires special kernels near the boundary \cite{bales93,ryan03,slingerland11,ryan15}. A disadvantage of the proposed method is that it requires a global operator, while the convolution with a Kernel is typically a local operator. This makes the new method unsuitable for problems that require an accurate reconstruction of the wave field throughout the entire time interval.

The paper is constructed as follows: first, we present the scalar wave equation and the corresponding classical finite element method in Section \ref{sec:model}. In Section \ref{sec:pp}, we then introduce the new pre- and post-processing algorithm. Super-convergence rates of the classical finite element method in negative-order Sobolev norms and of the post-processed solution in the energy- and $L^2$-norm are derived in Section \ref{sec:error}. In Section \ref{sec:tests}, we then demonstrate this super-convergence of the post-processed solution for numerous test cases, including cases with heterogeneous materials, unstructured meshes, and curved boundaries. Finally, we summarise our main conclusions in Section \ref{sec:conclusion}.

\section{Wave equation and finite element discretisation}
\label{sec:model}
Let $\Omega\subset\mathbb{R}^d$ be an open bounded $d$-dimensional domain with Lipschitz boundary $\partial\Omega$. We consider the scalar wave equation given by
\begin{subequations}
\begin{align}
\partial_t^2u &= \rho^{-1}\nabla\cdot c\nabla u + f &&\text{in }\Omega\times(0,T) \\
u&= 0 && \text{on }\partial\Omega\times (0,T), \\
u|_{t=0}&= u_0 &&\text{in }\Omega, \\
\partial_tu|_{t=0}&=v_0 &&\text{in }\Omega,
\end{align}
\label{eq:wave}%
\end{subequations} 
where $(0,T)$ is the time domain, $u:\Omega\times(0,T)\rightarrow\mathbb{R}$ is the wave field, $\nabla$ is the gradient operator, $u_0,v_0:\Omega\rightarrow\mathbb{R}$ are the initial wave and velocity field, $f:\Omega\times(0,T)\rightarrow\mathbb{R}$ is the source term, and $\rho,c:\Omega\rightarrow\mathbb{R}^+$ are strictly positive scalar fields that satisfy
\begin{align}
c_0\leq c\leq c_1 \qquad \text{and} \qquad \rho_0\leq\rho\leq\rho_1
\label{eq:bounds}
\end{align}
for some positive constants $\rho_0$, $\rho_1$, $c_0$, and $c_1$.

For the weak formulation, let $u_0\in H_0^1(\Omega)$, $v_0\in L^2(\Omega)$, and $f\in L^2(0,T;L^2(\Omega))$, where $L^2(\Omega)$ denotes the standard Lebesque space of square integrable functions on $\Omega$, $H^1_0(\Omega)$ denotes the standard Sobolev space of functions on $\Omega$ that have a zero trace on the boundary $\partial\Omega$ and have square-integrable weak derivatives, and $L^2(0,T;U)$, with $U$ a Banach space, is the standard Bochner space of functions $f:(0,T)\rightarrow U$, such that $\|f\|_U$ is square integrable on $(0,T)$. The weak formulation of \eqref{eq:wave} can then be formulated as finding $u\in L^{2}\big(0,T;H^1_0(\Omega)\big)$, with $\partial_tu\in L^2\big(0,T;L^2(\Omega)\big)$, $\partial_t(\rho\partial_tu)\in L^2\big(0,T;H^{-1}(\Omega)\big)$, such that $u|_{t=0}=u_0$, $\partial_tu|_{t=0}=v_0$, and 
\begin{align}
\label{eq:WF}
\langle \partial_t(\rho\partial_tu), w\rangle + a(u,w) &= (f,w)_{\rho} &&\text{for all }w\in H_0^1(\Omega), \text{a.e. }t\in(0,T),
\end{align}
where $H^{-1}(\Omega)$ denotes the dual space of $H^1_0(\Omega)$, $\langle \cdot,\cdot\rangle$ denotes the pairing between $H^{-1}(\Omega)$ and $H^1_0(\Omega)$, $(\cdot,\cdot)_{\rho}$ denotes the weighted $L^2(\Omega)$-product
\begin{align*}
(u,w)_{\rho} &= \int_{\Omega} \rho uw \;dx
\end{align*}
and $a(\cdot,\cdot)$ denotes the elliptic operator
\begin{align*}
a(u,w) &= \int_{\Omega} c\nabla u\cdot\nabla w \;dx.
\end{align*}
It can be shown, in a way analogous to \cite[Chapter 3, Theorem 8.1]{lions72}, that this problem is well posed and has a unique solution.

The weak formulation can be solved with the classical finite element method. Let $\mathcal{T}_h$ be a tessellation of $\Omega$ into straight and/or curved simplicial elements that all have a radius smaller than $h$. The degree-$p$ finite element space $U_h$ can then be constructed as follows:
\begin{align*}
U_h = U(\mathcal{T}_h,\hat U)&:= \{ u\in H^1_0(\Omega) \;|\; u\circ\phi_e\in\hat{U} \text{ for all }e\in\mathcal{T}_h\},
\end{align*}
where $\phi_e:\hat{e}\rightarrow e$ denotes the reference-to-physical element mapping and $\hat U:=\mathcal{P}_p(\hat e)$ denotes the reference element space, with $\mathcal{P}_p(\hat e)$ the space of polynomials on $\hat e$ of degree $p$ or less. The element mappings $\phi_e$ are usually affine mapping resulting in straight elements, but may also be higher-degree polynomial mappings resulting in curved elements that can better approximate domains with curved boundaries. The finite element method approximates the wave field by a discrete wave field $u_h:[0,T]:\rightarrow U_h$ that satisfies the initial conditions $u_h|_{t=0}=u_{0,h}$ and $\partial_tu_h|_{t=0}=v_{0,h}$ and that satisfies
\begin{align}
\label{eq:FEM}
(\partial_t^2u_h,w)_{\rho} + a(u_h,w) &= (f_h,w)_{\rho} &&\text{for all }w\in U_h, \text{a.e. }t\in(0,T).
\end{align}
The initial values $u_{0,h},v_{0,h}\in U_h$ and the discrete source term $f_h\in L^2\big(0,T;L^2(\Omega)\cap U_h\big)$ are projections of $u_0$, $v_0$, and $f$, onto the discrete space. We define $f_h := \Pih f$, where $\Pih$ is the weighted $L^2$-projection operator given by
\begin{align*}
(\Pih u,w)_{\rho} = (u,w)_{\rho} &&\text{for all }w\in U_h.
\end{align*} 
Typically, a (weighted) $L^2$-projection is also used to discretise the initial solutions, although other types of projections are introduced in the next section in order to obtain higher convergence rates.

\section{Pre- and post-processing}
\label{sec:pp}
To present the pre- and post-processing algorithm, we first define the differential operator $\Lc$ by
\begin{align*}
\Lc :=-\rho^{-1}\nabla\cdot c\nabla, 
\end{align*}
where $\Lc u\in L^2(\Omega)$ is well-defined when $u\in H^1(\Omega)$ and $\nabla\cdot c\nabla u\in L^2(\Omega)$. The inverse operator $\Lc^{-1}:L^2(\Omega)\rightarrow H_0^1(\Omega)$ is defined such that 
\begin{align*}
a(\Lc^{-1}f,w) &= (f,w)_{\rho} &&\text{for all }w\in H_0^1(\Omega).
\end{align*}
Note that $u=\Lc^{-1}f$ is the weak solution of the elliptic problem 
\begin{subequations}
\begin{align}
-\rho^{-1}\nabla\cdot c\nabla u &= f &&\text{in }\Omega, \\
u &= 0 &&\text{on }\partial\Omega.
\end{align}
\label{eq:ell}%
\end{subequations}
The existence of a unique weak solution follows from the fact that $a$ is coercive, which in turn follows from the Poincar\'e inequality and the boundedness of $c$. Also note that, for any $f\in L^2(\Omega)$, we have $\Lc\Lc^{-1}f=f$ and, for any $u\in H_0^1(\Omega)$ with $\Lc u\in L^2(\Omega)$, we have $\Lc^{-1}\Lc u=u$. We will also often use the notation $\Lc^{\alpha}$ and $\Lc^{-\alpha}$ to denote the operator $\Lc$ and $\Lc^{-1}$ applied $\alpha$ times, respectively, , e.g. $\Lc^2u=\Lc(\Lc u)$, $\Lc^3=\Lc(\Lc(\Lc u))$, etc.

The discrete version of $\Lc$, denoted by $\Lh:U_h\rightarrow U_h$, is defined such that 
\begin{align*}
(\Lh u_h,w)_{\rho} &= a(u_h,w), && \text{for all }w\in U_h,
\end{align*}
and the inverse operator $\Lh^{-1}:U_h\rightarrow U_h$ such that
\begin{align*}
a(\Lh^{-1} f_h,w) &= (f_h,w)_{\rho}, && \text{for all }w\in U_h,
\end{align*}
Note that $u_h=\Lh^{-1} f_h$ is the standard finite element approximation of the solution $u$ to \eqref{eq:ell}.

It can be shown that, if the weak solution $u$ of \eqref{eq:wave} satisfies $\partial_t^2u\in L^2\big(0,T;L^2(\Omega)\big)$ and if $f\in L^2\big(0,T;L^2(\Omega)\big)$, then $u$ satisfies \eqref{eq:wave} almost everywhere and we can write
\begin{align}
\label{eq:wave1a}
\partial_t^2u + \Lc u &= f ,
\end{align}
while for the finite element approximation, we can write
\begin{align}
\label{eq:wave1b}
\partial_t^2u_h + \Lh u_h &= f_h .
\end{align}

When $u$ and $f$ are sufficiently regular, we can differentiate \eqref{eq:wave1a} and \eqref{eq:wave1b} with respect to time and obtain equalities of the form
\begin{align*}
\partial_t^{k+2}u + \Lc \partial_t^k u &= \partial_t^k f, \\
\partial_t^{k+2}u_h + \Lh \partial_t^k u_h &= \partial_t^k f_h,
\end{align*}
for $k\geq 0$. By reordering the terms, we can obtain equalities of the form
\begin{subequations}
\label{eq:wave2a}
\begin{align}
\partial_t^{k+2}u &= -\Lc \partial_t^ku + \partial_t^k f = r_{k}^{(1)}(\partial_t^k u),  \label{eq:wave2aa} \\
\partial_t^{k}u &= \Lc^{-1}(-\partial_t^{k+2}u + \partial_t^k f) = r_k^{(-1)}(\partial_t^{k+2} u), 
\end{align}
\end{subequations}
and
\begin{subequations}
\label{eq:wave2b}
\begin{align}
\partial_t^{k+2}u_h &= -\Lh \partial_t^ku_h + \partial_t^k f_h = r_{k,h}^{(1)}(\partial_t^k u_h), \\
\partial_t^{k}u_h &= \Lh^{-1}(-\partial_t^{k+2}u_h + \partial_t^k f_h) = r_{k,h}^{(-1)}(\partial_t^{k+2} u_h), 
\end{align}
\end{subequations}
for $k\geq 0$, where $r_{k}$, $r_k^{-1}$, $r_{k,h}$, and $r_{k,h}^{-1}$ are mappings defined as follows:
\begin{align*}
r_{k}^{(1)}(w) &:= -\Lc w + \partial_t^k f,  & r_k^{(-1)}(w) &:= \Lc^{-1}(-w + \partial_t^k f), \\
r_{k,h}^{(1)}(w) &:= -\Lh w + \partial_t^k f_h,  & r_{k,h}^{(-1)}(w) &:= \Lh^{-1}(-w + \partial_t^k f_h), 
\end{align*}
for $k\geq 0$. By applying \eqref{eq:wave2a} and \eqref{eq:wave2b} multiple times, we can obtain the following relations:
\begin{subequations}
\label{eq:wave3a}%
\begin{align}
\partial_t^{k+2\alpha}u &= (r_{k+2\alpha-2}^{(1)} \circ r_{k+2\alpha-4}^{(1)}\circ r_{k+2\alpha-6}^{(1)}\cdots \circ r_{k}^{(1)})(\partial_t^ku) =: r_{k}^{(\alpha)}(\partial_t^ku), \\
\partial_t^{k}u &= \big( r_k^{(-1)} \circ r_{k+2}^{(-1)} \circ r_{k+4}^{(-1)} \cdots \circ r_{k+2\alpha-2}^{(-1)} \big)(\partial_t^{k+2\alpha}u) =: r_{k}^{(-\alpha)}(\partial_t^{k+2\alpha}u),
\end{align}
\end{subequations}
and
\begin{subequations}
\label{eq:wave3b}%
\begin{align}
\partial_t^{k+2\alpha}u_h &= (r_{k+2\alpha-2,h}^{(1)}\circ r_{k+2\alpha-4,h}^{(1)}\circ r_{k+2\alpha-6,h}^{(1)}\cdots \circ r_{k,h}^{(1)} )(\partial_t^ku_h) =: r_{k,h}^{(\alpha)}(\partial_t^ku_h), \\
\partial_t^{k}u_h &= \big( r_{k,h}^{(-1)}\circ r_{k+2,h}^{(-1)}\circ r_{k+4,h}^{(-1)}\cdots \circ r_{k+2\alpha-2,h}^{(-1)} \big)(\partial_t^{k+2\alpha}u_h) =: r_{k,h}^{(-\alpha)}(\partial_t^{k+2\alpha}u_h),
\end{align}
\end{subequations}
for $k\geq 0, \alpha \geq 0$, with $r_{k}^{(0)}$ and $r_{k,h}^{(0)}$ the identity operators. Note that
\begin{align*}
\big(r_{k}^{(\alpha)}\circ r_k^{(-\alpha)}\big) (w) &= w, &&w\in L^2(\Omega), \\
\big(r_{k,h}^{(\alpha)}\circ r_{k,h}^{(-\alpha)}\big) (w) &= w, &&w\in U_h.
\end{align*}

To improve the convergence rate of the finite element approximation by $q\geq 1$ orders, assume $u$ and $f$ are sufficiently regular and discretise the initial values as follows:
\begin{subequations}
\label{eq:u0h}
\begin{align}
u_{0,h} &= \big(r^{(-\alpha)}_{0,h} \circ \Pi_h r^{(\alpha)}_0\big) (u_0), &&\alpha = \lceil q/2\rceil, \\
v_{0,h} &=  \big(r^{(-\beta)}_{1,h} \circ \Pi_h  r^{(\beta)}_1\big) (v_0), &&\beta = \lfloor q/2\rfloor.
\end{align}
\end{subequations}
To obtain an improved approximation $u_{h*}$ and $v_{h*}$ of the wave field $u$ and velocity $\partial_tu$ at time $T$, respectively, we compute
\begin{subequations}
\label{eq:ush}
\begin{align}
u_{h*}(T) &= \big(r^{(-\alpha)}_0 \circ r^{(\alpha)}_{0,h}\big) \big(u_{h}(T)\big) , &&\alpha = \lceil q/2\rceil, \\
v_{h*}(T) &= \big(r^{(-\beta)}_1 \circ r^{(\beta)}_{1,h}\big) \big(\partial_tu_{h}(T)\big), &&\beta = \lfloor q/2\rfloor.
\end{align} 
\end{subequations}
Values of $\alpha$ and $\beta$ for $q\leq 10$ are listed in Table \ref{tab:abck}. In case of sufficient regularity, we can then obtain error estimates of the form
\begin{subequations}
\label{eq:errb1}
\begin{align*}
\|\partial_tu(T)-v_{h*}(T)\|_0 + \|u(T)-u_{h*}(T)\|_1 &\leq Ch^{p+\min(p,q)}, \\
\|u(T)-u_{h*}(T)\|_0 &\leq Ch^{p+\min(p,q+1)}, 
\end{align*}
\end{subequations}
(see Corollary \ref{cor:conv3}), where $p$ denotes the degree of the polynomial approximation and $C$ is some positive constant that does not depend on the mesh resolution $h$. This implies, for example, that a convergence rate of order $2p$ in the energy norm can be obtained when choosing $q=p$.

\begin{table}[h]
\caption{Number of processing steps $\alpha$ and $\beta$ at initial and final time to improve the convergence rate by $q$ orders.}
\label{tab:abck}
\begin{center}
{\tabulinesep=0.5mm
\begin{tabu}{c | ccccc ccccc}
$q$			& 1 & 2 & 3 & 4 & 5 & 6 & 7 & 8 & 9 & 10 \\ \hline
$\alpha$		& 1 & 1 & 2 & 2 & 3 & 3 & 4 & 4 & 5 & 5 \\
$\beta$		& 0 & 1 & 1 & 2 & 2 & 3 & 3 & 4 & 4 & 5
\end{tabu}}
\end{center}
\end{table}

From \eqref{eq:wave3a} and \eqref{eq:wave3b}, it follows that the choice for the initial conditions is equivalent to setting
\begin{align*}
\partial_t^qu_h|_{t=0} &= \partial_t^qu|_{t=0}, \qquad \partial_t^{q+1}u_h|_{t=0} = \partial_t^{q+1}u|_{t=0}.
\end{align*}
It also follows from \eqref{eq:wave3a} and \eqref{eq:wave3b}, that the post-processed solution at time $T$ is equal to the exact solution when
\begin{align*}
\partial_t^qu(T) &= \partial_t^qu_h(T), \qquad \partial_t^{q+1}u(T) = \partial_t^{q+1}u_h(T).
\end{align*}

In case of no source term, so $f=0$, the pre- and post-processed approximations simplify to
\begin{align*}
u_{0,h} &= \Lh^{-\alpha}\Pih\Lc^{\alpha}u_0, & u_{h*}(T) &=\Lc^{-\alpha}\Pih\Lh^{\alpha}u_h(T), \\
v_{0,h} &= \Lh^{-\beta}\Pih\Lc^{\beta}v_0, & v_{h*}(T) &=\Lc^{-\beta}\Pih\Lh^{\beta}\partial_tu_h(T).
\end{align*}

To discretise the initial values $u_0$ and $v_0$, we need to respectively apply $\alpha$ and $\beta$ times the operator $\Lh^{-1}$, which means we need to respectively solve $\alpha$ and $\beta$ elliptic problems of the form in \eqref{eq:ell} numerically using the finite element method. The mesh and approximation space used to solve these elliptic problems are the same as those used for the time-stepping. To obtain the improved wave field $u_{h*}$ or velocity field $v_{h*}$ at the final time, we need to respectively apply $\alpha$ or $\beta$ times the operator  $\Lc^{-1}$, which requires solving elliptic problems of the form in \eqref{eq:ell} $\alpha$ or $\beta$ times exactly. Obtaining the exact solution, however, is usually not possible. Instead, we can approximate these elliptic problems with a degree-$(p+q)$ finite element method in order to maintain the improved convergence rates. 

In the end, both pre- and post-processing therefore require solving a set of at most $\alpha+\beta=q$ linear systems. These linear systems can be solved by a direct method or with an iterative method. When using an iterative method, good first approximations can be obtained from the unprocessed initial and final value, as shown in Section \ref{sec:tests}.

The pre- and post-processing algorithms presented here apply to wave problems with zero Dirichlet boundary conditions, but can also be applied to problems with Neumann boundary conditions or periodic domains when we replace $H_0^1(\Omega)$ by $H^1(\Omega)\cap\{1\}^{\perp}$, where $\{1\}^{\perp} := \{u\in L^2(\Omega) \;|\; (u,1)_{\rho} = 0\}$. Furthermore, the algorithms can also be applied to problems with non-zero boundary conditions, since the non-zero part can be absorbed in the source term $f$ and the initial conditions. In particular, let $\tilde{u}$ be any smooth function that satisfies the non-zero boundary condition and set $\tilde{f}:=\partial_t^2\tilde{u} - \rho^{-1}\nabla\cdot c\nabla\tilde{u}$. We can then write $u=\hat{u}+\tilde{u}$, where $\hat{u}$ is the solution of the wave equation \eqref{eq:wave} with zero boundary conditions, source term $\hat{f}=f-\tilde f$ and initial conditions $\hat{u}_0=u_0-\tilde{u}|_{t=0}$ and $\hat{v}_0=v_0-\partial_t\tilde{u}|_{t=0}$.

\section{Error analysis}
\label{sec:error}
In this section, we analyse the accuracy of the classical finite element method in negative-order Solobev norms and the accuracy of the pre- and post-processing algorithms in the $L^2$- and energy norm. We only consider the semi-discrete case, so not the time discretisation. The outline of the error analysis is as follows: we first obtain error estimates for the classical finite element approximation of the form
\begin{align*}
\|\partial_tu - \partial_tu_h\|_{-m*}+ \|u-u_h\|_{1-m*} &\leq Ch^{p+\min(p,m)}
\end{align*}
for $m\geq 0$  (see Theorem \ref{thm:conv1}), where $\|\cdot\|_{0*}:=\|\cdot\|_0$ and $\|\cdot\|_{1*}:=\|\cdot\|_1$ denote the standard $L^2$- and $H^1$-norm, respectively, and $\|\cdot\|_{-m*}$, for $m>0$, denotes the adapted negative-order Sobolev-norm defined in \eqref{eq:negnorm}. The theory is similar to that of \cite{douglas78}, except that in \cite{douglas78}, they require a very complex and non-standard discretisation of the initial values, while we prove in Theorem \ref{thm:conv1} that the estimates also hold when the initial values are discretised using a standard weighted $L^2$-projection.

We then show (see proof Theorem \ref{thm:conv2}) that $\partial_t^qu$ and $\partial_t^qu_h$ can also be written as a solution and finite element approximation of the wave equation, respectively, and that by discretising $u_{0,h}$ and $v_{0,h}$ as in \eqref{eq:u0h}, the error is bounded by
\begin{align*}
\|\partial_t^{q+1}u - \partial_t^{q+1}u_h\|_{-m*}+ \|\partial_t^qu-\partial_t^qu_h\|_{1-m*} &\leq Ch^{p+\min(p,m)}
\end{align*}
for $m\geq 0$.

Finally, we show that 
\begin{align*}
& \|\partial_tu -v_{h*}\|_{-m*}+ \|u-u_{h*}\|_{1-m*} = \\
& \qquad \|\partial_t^{q+1}u - \partial_t^{q+1}u_h\|_{-(q+m)*}+ \|\partial_t^{q}u-\partial_t^{q}u_h\|_{1-(q+m)*}
\end{align*}
for $m\geq 0$ (see Theorem \ref{thm:conv2}), where $u_{h*}$ and $v_{h*}$ are post-processed solutions, defined as in \eqref{eq:ush}. By choosing $m=1$ or $m=0$, we then obtain bounds of the form
\begin{align*}
\|\partial_tu-v_{h*}\|_0 + \|u-u_{h*}\|_1 &\leq Ch^{p+\min(p,q)}, \\
\|u-u_{h*}\|_0 &\leq Ch^{p+\min(p,q+1)}.
\end{align*}

The analysis presented in this section relies on certain regularity assumptions that typically only hold when the boundary of the domain is sufficiently smooth. We therefore do not restrict our analysis to polyhedral domains, but assume that $\Omega$ has a piecewise smooth boundary. An accurate finite element approximation will then require the use of curved elements and the leading constants in the error estimates will then also depend on the shape-regularity of these elements. We therefore define the shape-regularity of the mesh, denoted by $\gamma$, as follows:
\begin{align}
\label{eq:defmeshreg}
\gamma &:= \max_{e\in\mathcal{T}_h}\gamma_e :=  \max_{e\in\mathcal{T}_h} \left( \sup_{\vx\in e} h_e\|\nabla\phi_e^{-1}(\vx)\| + \sum_{2\leq|\alpha|\leq p+1} \sup_{\hat{\vx}\in \hat e} h_{e}^{-|\alpha|} \|D^{\alpha}\phi_e(\hat{\vx}) \| \right),
\end{align}
where $h_{e}$ is the diameter of $e$, $\|\cdot\|$ denotes the Euclidean norm, $D^{\alpha}:=\partial_1^{\alpha_1}\partial_2^{\alpha_2}\dots\partial_d^{\alpha_d}$ is a partial differential operator, $\partial_i$ is the partial derivative with respect to Cartesian coordinate $x_i$, and $|\valpha|:=\alpha_1+\alpha_2+\dots+\alpha_d$ is the order of the operator.

In case of a straight element, $\gamma_e$ is proportional to the commonly used ratio $h_e/\rho_e$, where $\rho_e$ denotes the diameter of the inscribed sphere of $e$. For curved elements, the regularity constant $\gamma$ remains uniformly bounded for a sequence of meshes when using a suitable parametrisation, such as the one described in \cite{lenoir86}. We present a simpler alternative to this algorithm in Section \ref{sec:testCircle}.

To simplify the error analysis, we will always assume that the mesh covers the domain exactly, so $\bigcup_{e\in\mathcal{T}_h}\overline{e} = \overline{\Omega}$. Furthermore, for the readability of the analysis, we will always let $C$ denote a constant that may depend on the domain $\Omega$, the regularity of the mesh $\gamma$, the spatial parameters $\rho$ and $c$, and the polynomial degree $p$, but does not depend on the mesh resolution $h$, time interval $(0,T)$, or the functions that appear in the inequality. We will also always assume that $h\leq h_0$, for some sufficiently small $h_0>0$ that does not depend on the time interval $(0,T)$ or the functions that appear in the inequality.

Let $H^k(\Omega)$, for $k\geq 1$, denote the standard Sobolev space of functions on $\Omega$ with square-integrable $k$th-order derivatives equipped with norm
\begin{align*}
\|u\|_{k} &:= \sum_{|\valpha|\leq k} \|D^{\valpha}u \|_{0},
\end{align*}
where $\|\cdot\|_0$ denotes the standard $L^2(\Omega)$-norm. Also, let $H^k_0(\Omega)$ denote the space of functions $u\in H^k(\Omega)$ such that the trace of the derivatives is zero on the boundary: $D^{\valpha}u|_{\partial \Omega}=0$, for all $|\valpha|\leq k-1$. We then define the following negative-order Sobolev norms for functions in $L^2(\Omega)$:
\begin{align*}
\|u\|_{-k} &:= \sup_{w\in H_0^{k}(\Omega) \setminus \{0\}}\frac{(u,w)}{\|w\|_{k}}. 
\end{align*}
We also introduce an adapted version of these negative order norms:
\begin{align}
\label{eq:negnorm}
\|u\|_{-k*} &:= \begin{cases}
\| \Lc^{-\alpha}u\|_0, &k=2\alpha, \\
\| \Lc^{-\alpha}u\|_1, &k=2\alpha-1.
\end{cases}
\end{align}
It is proven in Section \ref{sec:nnorm}, that, if $\rho$ and $c$ are sufficiently smooth, then
\begin{align*}
\|u\|_{-k} &\leq C\|u\|_{-k*}.
\end{align*}

Now let $U$ be a Banach space and let $L^{q}(0,T;U))$ denote the Bochner space of functions $f:(0,T)\rightarrow U$ such that $\|f(t)\|_U$ is in $L^q(0,T)$. We equip the space with norm
\begin{align*}
\|u\|_{q,U} &:= \begin{cases}
\left(\int_{0}^T \|u(t)\|_{U}^q \;dt \right)^{1/q}, & 1\leq q<\infty, \\
\esssup_{t\in (0,T)} \|u(t)\|_{U}, & q=\infty.
\end{cases} 
\end{align*}
We will use $\|\cdot\|_{\infty,k}$ as a short-hand notation for the $L^{\infty}(0,T;H^k(\Omega))$-norm and define $\|\cdot\|_{\infty,-k*}$, for $k\geq 1$, as follows:
\begin{align*}
\|u\|_{\infty,-k*} &:= \esssup_{t\in (0,T)} \|u(t)\|_{-k*}, &&u\in L^{\infty}(0,T;L^2(\Omega)).
\end{align*}

We will often make a regularity assumption of the following type: if $f\in H^k(\Omega)$, then $\Lc^{-1}f\in H^{k+2}(\Omega)$ and
\begin{align}
\|\mathcal{L}^{-1}f\|_{k+2} &\leq C\|f\|_{k} &&\text{ for all }k\leq K.
\label{eq:reg}
\end{align}
Such an assumption certainly holds when $\partial\Omega\in \mathcal{C}^{K+2}$, $c\in \mathcal{C}^{K+1}(\overline\Omega)$, and $\rho\in\mathcal{C}^K(\overline\Omega)$. 

We can now prove super-convergence for the finite element method combined with pre- and post-processing by proving the following lemma's and theorems.

\begin{lem}[error equation]
\label{lem:err}
Let $u$ be the solution of \eqref{eq:WF} and let $u_h$ be the solution of \eqref{eq:FEM} with $f_h:=\Pih f$. Assume $\partial_t^2u \in L^2\big(0,T;L^2(\Omega)\big)$. Then the discrete error $e_h:=\Pih u-u_h$ satisfies
\begin{align}
\label{eq:err}
(\partial_t^2e_h,w)_{\rho} + a(e_h,w) &= - a(\epsilon_h,w)
\end{align}
for all $w\in U_h$ and almost every $t\in(0,T)$, where $\epsilon_h:=u-\Pih u$ is the projection error.
\end{lem}

\begin{proof}
By subtracting \eqref{eq:FEM} from \eqref{eq:WF} and by using the assumption that $\partial_t^2u \in L^2\big(0,T;L^2(\Omega)\big)$, we obtain
\begin{align*}
(\partial_t^2(u-u_h),w)_{\rho} + a(u-u_h,w) &= (f-f_h,w)_{\rho}
\end{align*}
for all $w\in U_h$, almost every $t\in(0,T)$. Substituting $u-u_h = e_h+\epsilon_h$ and reordering the terms gives
\begin{align*}
(\partial_t^2e_h,w)_{\rho} + a(e_h,w) &=  - (\partial_t^2\epsilon_h,w)_{\rho} - a(\epsilon_h,w) + (f-f_h,w)_{\rho}
\end{align*}
for all $w\in U_h$, almost every $t\in(0,T)$. By definition of $\Pi_h$, $ (\partial_t^2\epsilon_h,w)_{\rho}=0$ and $(f-f_h,w)_{\rho}=(f -\Pih f,w)_{\rho} = 0$ for all $w\in U_h$. The above then results in \eqref{eq:err}.
\end{proof}

\begin{thm}
\label{thm:conv1}
Let $u$ be the weak solution of \eqref{eq:WF} and let $u_h$ be the degree-$p$ finite element approximation of \eqref{eq:FEM}, with $p\geq 1$ and with $u_{0,h}:=\Pih u_0$, $v_{0,h}:=\Pih v_0$, and $f_h:=\Pih f$. Assume regularity condition \eqref{eq:reg} holds for some $K\geq 0$ and assume $u,\partial_tu\in L^{\infty}\big(0,T;H^{k}(\Omega)\big)$ for some $k\geq 1$ and $\partial_t^2u\in L^{\infty}\big(0,T;L^2(\Omega)\big)$. Then
\begin{align}
&\qquad \|\partial_tu-\partial_tu_h\|_{\infty,-m*}+\|u-u_h\|_{\infty,1-m*} \leq  \label{eq:conv1} \\
& Ch^{\min(p,k-1)+\min(p,m,K+1)}(\|u\|_{\infty,\min(p+1,k)} + T\|\partial_tu\|_{\infty,\min(p+1,k)}) \nonumber
\end{align}
for any $m\geq 1$.
\end{thm}

\begin{proof}
We first consider the case $m=2\alpha$, with $\alpha\geq 0$. Define discrete error $e_h:=\Pih u-u_h$ and projection error $\epsilon_h:=u-\Pih u$ and use Lemma \ref{lem:err} to obtain
\begin{align}
\label{eq:conv1a}
(\partial_t^2e_h,w)_{\rho} + a(e_h,w) &= - a(\epsilon_h,w)
\end{align}
for all $w\in U_h$, almost every $t\in(0,T)$. Set $w=\Lh^{-2\alpha}\partial_te_h$ and use Lemma \ref{lem:L1} to obtain
\begin{align*}
\partial_tE_h &=  R_h &&\text{for a.e. }t\in(0,T),
\end{align*} 
where
\begin{align*}
E_h &:= \frac12(\Lh^{-\alpha}\partial_te_h,\Lh^{-\alpha}\partial_te_h)_{\rho} + \frac12a(\Lh^{-\alpha}e_h,\Lh^{-\alpha}e_h), \\
R_{h} &:=- \partial_ta(\epsilon_h,\Lh^{-2\alpha}e_h) + a(\partial_t\epsilon_h,\Lh^{-2\alpha}e_h).
\end{align*}
Integrating over $(0,T')$, with $T'\in(0,T)$, results in 
\begin{align}
E_h|_{t=T'} &= E_h|_{t=0} + \int_0^{T'} R_h \;dt &&\text{for a.e. }T'\in(0,T).
\label{eq:conv1c}
\end{align}
From the boundedness of $\rho$ and the coercivity of $a$, it follows that
\begin{align}
\|\Lh^{-\alpha}\partial_te_h\|_{0} + \|\Lh^{-\alpha}e_h\|_{1} &\leq C\sqrt{E_h} , &&\text{for a.e. }t\in(0,T).
\label{eq:conv1d}
\end{align}
To bound $R_h$, we derive the following:
\begin{align*}
|a(\partial_t^i\epsilon_h,\Lh^{-2\alpha}e_h)| 
&\leq |a\big(\partial_t^i\epsilon_h,(\Lh^{-\alpha}-\Lc^{-\alpha})\Lh^{-\alpha}e_h\big)| + |a\big(\partial_t^i\epsilon_h,\Lc^{-\alpha}\Lh^{-\alpha}e_h\big)| \\
&= |a\big(\partial_t^i\epsilon_h,(\Lh^{-\alpha}-\Lc^{-\alpha})\Lh^{-\alpha}e_h\big)| + |a\big(\Lc^{-\alpha}\partial_t^i\epsilon_h,\Lh^{-\alpha}e_h\big)| \\
&\leq Ch^{\min(p,2\alpha,K+1)} \|\partial_t^i\epsilon_h\|_1 \|\Lh^{-\alpha}e_h\|_{1} + C\|\Lc^{-\alpha}\partial_t^i\epsilon_h\|_1\|\Lh^{-\alpha}e_h\|_{1} \\
&\leq Ch^{\min(p,k-1)+\min(p,2\alpha,K+1)} \|\partial_t^iu\|_{\min(p+1,k)}\|\Lh^{-\alpha}e_h\|_{1} 
\end{align*}
for $i=0,1$ and almost every $t\in(0,T)$. Here, the first line follows from the triangle inequality, the second line follows from Lemma \ref{lem:L1}, the third line follows from the Cauchy--Scwarz inequality and Lemma \ref{lem:ell4}, and the last line follows from Lemma \ref{lem:int} and Lemma \ref{lem:int2}. Using this inequality, we can obtain
\begin{align}
& \qquad \left | \int_0^{T'} R_h \;dt \right| \leq \label{eq:conv1e} \\
&Ch^{\min(p,k-1)+\min(p,2\alpha,K+1)}\big(\|u\|_{\infty,\min(p+1,k)}+T\|\partial_t u\|_{\infty,\min(p+1,k)}\big) \|\Lh^{-\alpha}e_h\|_{\infty,1} \nonumber
\end{align}
for almost every $T'\in(0,T)$. Since $u_{h,0}=\Pih u_0$ and $v_{h,0}=\Pih v_0$, we have $e_h|_{t=0}=0$ and $\partial_te_h|_{t=0} =0$ and therefore
\begin{align}
E_h|_{t=0} &= 0 \label{eq:conv1g}.
\end{align}
Taking the essential supremum of \eqref{eq:conv1c} over all $T'\in(0,T)$ and using \eqref{eq:conv1d}, \eqref{eq:conv1e}, and \eqref{eq:conv1g} results in
\begin{align}
 \label{eq:conv1h}
& \qquad \|\Lh^{-\alpha}\partial_te_h\|_{\infty,0}+\|\Lh^{-\alpha}e_h\|_{\infty,1} \leq \\
& Ch^{\min(p,k-1)+\min(p,2\alpha,K+1)}(\|u\|_{\infty,\min(p+1,k)} + T\|\partial_tu\|_{\infty,\min(p+1,k)}). \nonumber
\end{align}
Using this result, we can derive
\begin{align*}
& \qquad\qquad \|\Lc^{-\alpha}\partial_te_h\|_{\infty,0}+\|\Lc^{-\alpha}e_h\|_{\infty,1} \leq  \\
&\leq \|(\Lc^{-\alpha}-\Lh^{-\alpha})\partial_te_h\|_{\infty,0} + \|(\Lc^{-\alpha}-\Lh^{-\alpha})e_h\|_{\infty,1} + \|\Lh^{-\alpha}\partial_te_h\|_{\infty,0}  + \|\Lh^{-\alpha}e_h\|_{\infty,1} \\
&\leq Ch^{\min(p,2\alpha,K+1)} ( \| \partial_te_h\|_{\infty,0}  + \|e_h\|_{\infty,1}) + \|\Lh^{-\alpha}\partial_te_h\|_{\infty,0}  + \|\Lh^{-\alpha}e_h\|_{\infty,1} \\
&\leq Ch^{\min(p,k-1)+\min(p,2\alpha,K+1)} (\|u\|_{\infty,\min(p+1,k)} + T\|\partial_tu\|_{\infty,\min(p+1,k)}).
\end{align*}
Here, we used the triangle inequality in the second line, Lemma \ref{lem:ell4} in the third line, and \eqref{eq:conv1h} in the last line. Since $u-u_h = e_h+\epsilon_h$, \eqref{eq:conv1} then follows from the above and Lemma \ref{lem:int2}.

In case $m=2\alpha+1$, we choose $w$ in \eqref{eq:conv1a} as $w=\Lh^{-(2\alpha+1)}$. Using Lemma \ref{lem:L1}, we can then derive 
\begin{align*}
\partial_tE_h &=  R_h &&\text{for a.e. }t\in(0,T),
\end{align*} 
where
\begin{align*}
E_h &:= \frac12a(\Lh^{-(\alpha+1)}\partial_te_h,\Lh^{-(\alpha+1)}\partial_te_h) + \frac12(\Lh^{-\alpha}e_h,\Lh^{-\alpha}e_h)_{\rho}, \\
R_{h} &:=- \partial_ta(\epsilon_h,\Lh^{-(2\alpha+1)}e_h) + a(\partial_t\epsilon_h,\Lh^{-(2\alpha+1)}e_h).
\end{align*}
The rest of the proof is analogous to the proof for the case $m=2\alpha$.
\end{proof}

\begin{lem}
\label{lem:r2}
Let $w\in L^{\infty}\big(0,T;L^2(\Omega)\big)$ and $w_h\in L^{\infty}\big(0,T; U_h\big)$ be two arbitrary functions. Also, let $k\geq 0$ and $\alpha\geq 1$ and assume that $\partial_t^{i}f \in L^{\infty}\big(0,T;L^2(\Omega)\big)$ for all $i\leq k+2\alpha-2$. Then
\begin{subequations}
\label{eq:r2b}%
\begin{align}
r_k^{(-\alpha)}(w) &= (-1)^\alpha\Lc^{-\alpha}w + \sum_{\alpha'=1}^{\alpha} (-1)^{(\alpha'-1)} \Lc^{-\alpha'}\partial_t^{k+2(\alpha'-1)}f, \\
r_{k,h}^{(-\alpha)}(w_h) &= (-1)^\alpha\Lh^{-\alpha}w_h + \sum_{\alpha'=1}^{\alpha} (-1)^{(\alpha'-1)} \Lh^{-\alpha'}\partial_t^{k+2(\alpha'-1)}f_h. 
\end{align}
\end{subequations}
\end{lem}

\begin{proof}
The results follow readily from induction on $\alpha$.
\end{proof}

\begin{thm}
\label{thm:conv2}
Let $u$ be the weak solution of \eqref{eq:WF} and let $u_h$ be the degree-$p$ finite element approximation of \eqref{eq:FEM}, with $p\geq 1$ and with $f_h:=\Pih f$. Assume regularity condition \eqref{eq:reg} holds for some $K\geq 0$. Also, let $q\geq 0$ and $k\geq 1$ and assume that $\partial_t^{i}u\in L^{\infty}\big(0,T;H^{k}(\Omega)\big)$ for all $i\leq q+1$, $\partial_t^{q+2}u \in L^{\infty}\big(0,T;L^2(\Omega)\big)$ and $\partial_t^{i}f \in L^{\infty}\big(0,T;L^2(\Omega)\big)$ for all $i\leq q$. If the initial conditions are discretized by
\begin{align*}
u_{0,h} &= \big(r^{(-\alpha)}_{0,h} \circ \Pi_h r^{(\alpha)}_0\big) \big(u_0\big), &&\alpha = \lceil q/2\rceil, \\
v_{0,h} &= \big(r^{(-\beta)}_{1,h} \circ \Pi_h  r^{(\beta)}_1\big) \big(v_0\big), &&\beta = \lfloor q/2\rfloor,
\end{align*}
and if the post-processed solution is computed by
\begin{align*}
u_{h*}(t) &= \big(r^{(-\alpha)}_0 \circ r^{(\alpha)}_{0,h}\big) \big(u_{h}(t)\big) , &&\alpha = \lceil q/2\rceil, \\
v_{h*}(t) &= \big(r^{(-\beta)}_1 \circ r^{(\beta)}_{1,h}\big) \big(\partial_tu_{h}(t)\big) , &&\beta = \lfloor q/2\rfloor,
\end{align*}
then
\begin{align}
& \qquad \|\partial_tu-v_{h*}\|_{\infty,-m*}+\|u-u_{h*}\|_{\infty,1-m*} \leq \label{eq:conv2a} \\
& Ch^{\min(p,k-1)+\min(p,q+m,K+1)}(\|\partial_t^q u\|_{\infty,\min(p+1,k)} + T\|\partial_t^{q+1}u\|_{\infty,\min(p+1,k)}) \nonumber
\end{align}
for all $m\geq 0$.
\end{thm}

\begin{proof}
From the regularity of the time derivatives of $u$ and $f$, it follows that
\begin{subequations}
\label{eq:dt1}%
\begin{align}
\partial_t^{q+2}u + \Lc \partial_t^q u &= \partial_t^q f, \label{eq:dt1a} \\
\partial_t^{q+2}u_h + \Lh \partial_t^q u_h &= \partial_t^q f_h. \label{eq:dt1b}
\end{align}
\end{subequations}
for a.e. $(x,t)\in\Omega\times(0,T)$. From the definitions of $r_i^{(\kappa)}$ and $r_{i,h}^{(\kappa)}$, it also follows that
\begin{subequations}
\label{eq:r1}%
\begin{align}
\partial_t^{i+2\kappa}u&=r_i^{(\kappa)}(\partial_t^iu),   \label{eq:r1a} \\
\partial_t^{i+2\kappa}u_h&=r_{i,h}^{(\kappa)}(\partial_t^iu_h), \label{eq:r1b}
\end{align}
\end{subequations}
for $i+2\kappa\leq q+1$.

Now, define $u^{(q)}:=\partial_t^qu$ and $u^{(q)}_h:=\partial_t^qu_h$. From \eqref{eq:dt1}, it follows that $u^{(q)}$ is the solution of the weak formulation given in \eqref{eq:WF} when we replace $u_0$, $v_0$, and $f$, by $u^{(q)}_0:=\partial_t^qu|_{t=0}$, $u^{(q+1)}_0:=\partial_t^{q+1}u|_{t=0}$, and $f^{(q)}:=\partial_t^q f$, respectively, and $u^{(q)}_h$ is the finite element approximation given in \eqref{eq:FEM} when we replace $u_{0,h}$, $v_{0,h}$, and $f_h$, by $u^{(q)}_{0,h}:=\partial_t^qu_h|_{t=0}$, $u^{(q+1)}_{0,h}:=\partial_t^{q+1}u_h|_{t=0}$, and $f^{(q)}_h:=\partial_t^q f_h=\Pih f^{(q)}$, respectively. 

We now prove that, from the discretisation of $u_{h,0}$ and $v_{h,0}$, it follows that
\begin{align}
\label{eq:conv2c}
u^{(q)}_{0,h} = \Pih u^{(q)}_{0} \text{ and } u^{(q+1)}_{0,h} = \Pih u^{(q+1)}_{0} .
\end{align}
To prove this, we first consider the case $q=2\alpha$. Then $\beta=\alpha$ and we can derive 
\begin{align*}
u^{(q)}_{0,h} = r_{0,h}^{(\alpha)}(u_{0,h}) = (r_{0,h}^{(\alpha)} \circ r^{(-\alpha)}_{0,h} \circ \Pih r^{(\alpha)}_0) (u_0) = \Pih r_{0}^{(\alpha)}(u_{0}) = \Pih u^{(q)}_{0}.
\end{align*}
Here, the first equality follows from \eqref{eq:r1b}, the second equality from the discretisation of $u_{h,0}$, and the last equality from \eqref{eq:r1a}. In an analogous way, we can show that
\begin{align*}
u^{(q+1)}_{0,h} = r_{1,h}^{(\beta)}(v_{0,h}) = (r_{1,h}^{(\beta)} \circ r^{(-\beta)}_{1,h} \circ \Pih  r^{(\beta)}_1) (v_0)   = \Pih r_{1}^{(\beta)}(v_{0}) = \Pih u^{(q+1)}_{0}.
\end{align*}
The proof of \eqref{eq:conv2c} for the case $q=2\alpha-1$ is analogous to that for the case $q=2\alpha$. From Theorem \ref{thm:conv1}, it then follows that
\begin{align}
\label{eq:conv2d}
&\qquad \|u^{(q+1)}-u^{(q+1)}_h\|_{\infty,-m'*}+\|u^{(q)}-u^{(q)}_h\|_{\infty,1-m'*} \leq \\
& Ch^{\min(p,k-1)+\min(p,m',K+1)}(\|u^{(q)}\|_{\infty,\min(p+1,k)} + T\|u^{(q+1)}\|_{\infty,\min(p+1,k)})  \nonumber
\end{align}
for all $m'\geq 0$. 

Next, we prove that
\begin{align}
\label{eq:conv2e}
& \|\partial_tu - v_{h*}\|_{-m*} + \|u-u_{h*}\|_{1-m*} =  \\
& \qquad  \|u^{(q+1)}-u^{(q+1)}_h\|_{-(q+m)*} + \|u^{(q)}-u^{(q)}_h\|_{1-(q+m)*}, \nonumber
\end{align}
for $m\geq 0$, a.e. $t\in(0,T)$. We first consider the case $q=2\alpha$. Then $\beta=\alpha$ and we can derive
\begin{align*}
\|u-u_{h*}\|_{1-m*} &= \|u-(r_{0}^{(-\alpha)}\circ r_{0,h}^{(\alpha)})(u_h) \|_{1-m*} \\
&=  \| r_{0}^{(-\alpha)}(u^{(q)})-r_{0}^{(-\alpha)}(u^{(q)}_h) \|_{1-m*} \\
&= \|\Lc^{-\alpha}(u^{(q)}-u^{(q)}_h) \|_{1-m*} \\
&= \|u^{(q)}-u^{(q)}_h\|_{1-(q+m)*}.
\end{align*}
Here, the first equality follows from the definition of $u_{h*}$, the second equality follows from \eqref{eq:r1}, the third equality follows from Lemma \ref{lem:r2}, and the last equality follows from the definitions of the adapted negative-order Sobolev norm. In an analogous way, we can derive
\begin{align*}
\|\partial_tu-v_{h*}\|_{-m*} &= \|\partial_tu-(r_{1}^{(-\beta)} \circ r_{1,h}^{(\beta)})(\partial_tu_h) \|_{-m*} \\
&=  \| r_{1}^{(-\beta)}(u^{(q+1)})- r_{1}^{(-\beta)}(u^{(q+1)}_h) \|_{-m*} \\
&= \|\Lc^{-\beta}(u^{(q+1)}-u^{(q+1)}_h) \|_{-m*} \\
&= \| u^{(q+1)}-u^{(q+1)}_h \|_{-(q+m)*}.
\end{align*}
Together, these last two equalities result in \eqref{eq:conv2e} for the case $q=2\alpha$. The proof for the case $q=2\alpha-1$ is analogous to that for $q=\alpha$. Inequality \eqref{eq:conv2a} then follows immediately from \eqref{eq:conv2d} and \eqref{eq:conv2e} when we set $m'=q+m$.
\end{proof}

By taking $m=0$ and $m=1$ in Theorem \ref{thm:conv2}, we immediately obtain the following.

\begin{cor}
\label{cor:conv3}
Let $u$ be the weak solution of \eqref{eq:WF}, let $u_h$ be the degree-$p$ finite element approximation of \eqref{eq:FEM}, with $p\geq 1$ and with $f_h$, $u_{0,h}$, $v_{0,h}$, $u_{h*}$, and $v_{h*}$ chosen as in Theorem \ref{thm:conv2}. If the assumptions in Theorem\ref{thm:conv2} hold for a certain $q\geq 0$ and for $k=p+1$ and $K=p-1$, then
\begin{align}
\|\partial_tu-v_{h*}\|_{\infty,0}+\|u-u_{h*}\|_{\infty,1} &\leq Ch^{p+\min(p,q)}(\|\partial_t^q u\|_{\infty,p+1} + T\|\partial_t^{q+1}u\|_{\infty,p+1})  \label{eq:conv3a}
\end{align}
and
\begin{align}
\|u-u_{h*}\|_{\infty,0} &\leq Ch^{p+\min(p,q+1)}(\|\partial_t^q u\|_{\infty,p+1} + T\|\partial_t^{q+1}u\|_{\infty,p+1}).  \label{eq:conv3b}
\end{align}
\end{cor}

\begin{rem}
While the leading constant in Theorem \ref{thm:conv2} depends on the shape-regularity of the elements, the theory is not restricted to (quasi-)uniform meshes but holds for general unstructured meshes.
\end{rem}

\section{Implementation and numerical tests}
\label{sec:tests}

\subsection{Time discretisation}
In the error analysis, we considered exact integration in space and time. In practice, however, we also need to discretise in time and use numerical integration to evaluate the spatial integrals and lump the mass matrix. 

To discretise in time, we first rewrite the finite element formulation given in \eqref{eq:FEM} as a system of ODE's. To do this, we use nodal basis functions. Let $\mathcal{Q}_h=\{\vx_i\}_{i=1}^N$ be a set of nodes of the form
\begin{align*}
\mathcal{Q}_h=\mathcal{Q}(\mathcal{T}_h,\hat{\mathcal{Q}}_I) \setminus \partial\Omega := \left(\bigcup_{e\in\mathcal{T}_h, \hat\vx\in\hat{\mathcal{Q}}} \phi_e(\hat\vx) \right) \setminus \partial\Omega,
\end{align*}
where $\hat{\mathcal{Q}}_I$ are the nodes on the reference element, and let $\{w_i\}_{i=1}^N$ be the nodal basis functions that span the discrete space $U_h$ and satisfy $w_i(\vx_j)=\delta_{ij}$, for $i,j=1,\dots,N$, with $\delta$ the Kronecker delta. Also, for any $f\in \mathcal{C}^0(\overline\Omega)$, define the vector $\vvct{f}\in\mathbb{R}^N$ as $\vvct{f}_i:=f(\vx_i)$ for $i=1,\dots,N$. The finite element formulation can then be written as finding $\vvct{u_h}:[0,T]\rightarrow\mathbb{R}^n$, such that $\vvct{u_h}|_{t=0}=\vvct{u_{0,h}}$, $\partial_t\vvct{u_h}|_{t=0}=\vvct{v_{0,h}}$, and 
\begin{align}
\partial_t^2\vvct{u_h} + L_h\vvct{u_h} &= \vvct{f_h}, && \text{for a.e. }t\in(0,T),
\end{align}
where $L_h:=M^{-1}A$ and $M,A\in\mathbb{R}^{N\times N}$ are the mass matrix and stiffness matrix, respectively, given by
\begin{align*}
M_{ij} = (w_i,w_j)_{\rho}, \qquad
A_{ij} = a(w_i,w_j).
\end{align*}

For the time discretisation, let $\Delta t$ denote the time step size, let $t_n:=n\Delta t$, and let $\vvu_h^n$ denote the approximation of $\vvct{u_h}(t_n)$. In order to maintain an order-$2p$ convergence rate, we use an order-$2p$ Dablain scheme \cite{dablain86} for time-stepping, which is given by
\begin{align*}
\vvu_h^{n+1} &= -\vvu_h^{n-1} + 2\vvu_h^n +  2\sum_{\alpha=1}^p \frac{\Delta t^{2\alpha}}{(2\alpha)!} \vvct{\partial_t^{2\alpha}u}_h^n, &&n\geq 1,
\end{align*}
where $\vvct{\partial_t^{k}u}_h^n$ is recursively defined by
\begin{align}
\label{eq:wave4a}
\vvct{\partial_t^{k+2}u}_h^n &:=  -L_h\left(\vvct{\partial_t^ku}_h^n\right) + \partial_t^k\vvct{f_h}(t_n) ,&&k\geq 0, n\geq 0
\end{align}
so the approximation $\vvu_h^{n+1}$ is computed using the two previous approximations $\vvu_h^{n-1}$ and $\vvu_h^{n}$. For the first time step, the computations are as follows:
\begin{align*}
\vvu_h^1 &= \vvu_h^0 + \Delta t\vvct{\partial_tu}_h^0 + \sum_{\alpha=1}^p \frac{\Delta t^{2\alpha}}{(2\alpha)!} \vvct{\partial_t^{2\alpha}u}_h^0 + \sum_{\alpha=1}^p \frac{\Delta t^{2\alpha+1}}{(2\alpha+1)!} \vvct{\partial_t^{2\alpha+1}u}_h^0,
\end{align*}
with $\vvu_h^0:= \vvct{u_{0,h}}$ and $\vvct{\partial_tu}_h^0:=\vvct{v_{0,h}}$. This scheme is commonly used for wave propagation modelling and has the advanatge that it only requires $p$ stages to obtain an order-$2p$ convergence rate. 

Dablain's scheme only computes the displacement field. A second-order approximation of the velocity $\underline{\partial_tu_h}$ at some time $t_n$ can be obtained by 
\begin{align*}
\vvv_h^{n,2} &:= \frac{\vvu_h^{n+1}-\vvu_h^{n-1}}{2\Delta t}.
\end{align*}
In case $\vvu_h^{n+1}=\vvct{u_h}(t_{n+1})$, $\vvu_h^{n-1}=\vvct{u_h}(t_{n-1})$, and $\vvct{\partial_tu}_h^n=\vvct{\partial_tu_h}(t_n)$, the Taylor approximation of $\vvv_h^{n,2}$ around $t_n$ is given by
\begin{align*}
\vvv_h^{n,2} &= \vvct{\partial_tu}_h^n + \sum_{\alpha=1}^{p-1}\frac{\Delta t^{2\alpha}}{(2\alpha+1)!} \vvct{\partial_t^{2\alpha+1}u}_h^n + \mathcal{O}(\Delta t^{2p}).
\end{align*}
Using this expansion, we can obtain higher-order approximations of the velocity. For example, fourth- and sixth-order approximations are given by
\begin{align*}
\vvv_h^{n,4} &:= \vvv_h^{n,2} + \Delta t^2\left(\frac{1}{6}L_h\vvv_h^{n,2} - \frac{1}{6}\partial_t\vvct{f_h}(t_n)\right), \\
\vvv_h^{n,6} &:= \vvv_h^{n,4} + \Delta t^4\left(\frac{7}{360}L_h^2\vvv_h^{n,2} - \frac{7}{360}L_h\partial_t \vvct{f_h}(t_n) - \frac{1}{120}{\partial_t^3 \vvct{f_h}(t_n)} \right).
\end{align*}
Since we are only interested in the order-$2p$ accurate velocity field at the final time slot $n=N_T:=T/\Delta t$, we only need to do this computation once.

For stability, the time step size should be sufficiently small:
\begin{align*}
\Delta t \leq \Delta t_{\max} = \sqrt{\frac{c_p}{\sigma_{\max}(L_h)}},
\end{align*}
where $c_p=4,12,7.57$ \cite{geevers18a}, for $p=1,2,3$, respectively, and $\sigma_{\max}(L_h)$ denotes the largest eigenvalue of $L_h$. The largest eigenvalue can be bounded by the largest eigenvalues of the element matrices \cite{irons71}:
\begin{align*}
\sigma_{\max}(L_h) &\leq \max_{e\in\mathcal{T}_h} \;\sigma_{\max}(M_e^{-1}A_e) ,
\end{align*}
where $M_e, A_e$ denote the mass- and stiffness matrix of element $e$, respectively. In the numerical tests, we will always set $\Delta t = 0.9\Delta t_{\max}$, with $\sigma_{\max}(L_h)$ estimated using the above.

When pre-processing, the initial discrete values $\vvct{u_{0,h}}=\vvu_h^0$ and $\vvct{v_{0,h}}=\vvct{\partial_tu}_h^0$ are computed by
\begin{align*}
\vvct{\partial_t^qu}_h^0 = \vvct{\partial_t^{q}u|_{t=0}}, \qquad \vvct{\partial_t^{q+1}u}_h^0 =  \vvct{\partial_t^{q+1}u|_{t=0}},
\end{align*}
and by recursively solving
\begin{align}
\label{eq:wave4b}
\vvct{\partial_t^{k}u}_h^0 =  L_h^{-1}\left(-\vvct{\partial_t^{k+2}u}_h^0 + {\partial_t^k \vvct{f_h} (t_n)} \right), &&k = q-1,q-2,\dots,1,0.
\end{align}
The initial values $\partial_t^{q}u|_{t=0}$ and $\partial_t^{q+1}u|_{t=0}$ can be obtained from $u_0$ and $v_0$ by computing \eqref{eq:wave2aa} for $k=0,1,\dots,q-1$.

When post-processing, we need to apply operators of the form $r_k^{(-1)}$ and therefore apply the operator $\Lc^{-1}f$, which maps a function $f$ to the exact solution of the elliptic problem given in \eqref{eq:ell}. Since it is usually not possible to solve the elliptic problem exactly, we approximate $\Lc^{-1}$ by $\Lhs^{-1}$, which maps a function $f$ to a finite element approximation of the elliptic problem. For this finite element approximation, we use the same mesh as for the time-stepping but with a degree-$(p+q)$ finite element space.

Let $U_{h*}$ denote the higher-degree finite element space, $\mathcal{Q}_{h*}=\{\vx_{i*}\}_{i=1}^{N_*}$ the corresponding nodes, $M_*$, $A_*$, and $L_{h*}:=M^{-1}_*A_*$ the corresponding matrices, and $\vvct{f}_*$ the vector of values of a function $f$ at the nodes $\vx_{i*}$. Also, let $P\in\mathbb{R}^{N_*\times N}$, defined by $P_{ij}:=w_j(\vx_{i*})$, be the matrix that maps the degrees of freedom of the degree-$p$ finite element space to the degrees of freedom of the degree-$(p+q)$ finite element space. The post-processed wave field $\vvu_{h*}^n$ and velocity field $\vvv_{h*}^n:=\vvct{\partial_tu}_{h*}^n$ can then be computed by
\begin{align*}
\vvct{\partial_t^{q}u}_{h*}^n = P\vvct{\partial_t^{q}u}_h^n, \qquad \vvct{\partial_t^{q+1}u}_{h*}^n =  P\vvct{\partial_t^{q+1}u}_{h}^n,
\end{align*}
and by recursive solving
\begin{align}
\label{eq:wave4c}
\vvct{\partial_t^{k}u}_{h*}^n =  L_{h*}^{-1}\left(-\vvct{\partial_t^{k+2}u}_{h*}^n + {\partial_t^k\vvct{f_h}_*(t_n)}\right), && k = q-1,q-2,\dots,1,0.
\end{align}
The discrete time derivatives $\vvct{\partial_t^{q}u}_h^n$ and $\vvct{\partial_t^{q+1}u}_h^n$ can be obtained from $\vvu_h^n$ and $\vvct{\partial_tu}_h^n:=\vvv_h^{n,2p}$ by recursively computing \eqref{eq:wave4a} for $k=0,1,\dots,q-1$. Again, we only need to do these post-processing steps at the final time slot $n=N_T$.

\subsection{Quadrature rules and mass lumping}
To compute the spatial integrals, we use a quadrature rule that consists of a set of points $\hat{\mathcal{Q}}_M$ on the reference element and a set of weights $\{\omega_{\hat\vx}\}_{\hat\vx\in\hat{\mathcal{Q}}_M}$. The integral over the reference element is then approximated as follows:
\begin{align*}
\int_{\hat e} f(\hat\vx) \;d\hat x &\approx \sum_{\hat\vx\in\hat{\mathcal{Q}}_M} \omega_{\hat\vx} f(\hat\vx).
\end{align*}
We can write integrals over the physical element as integrals over the reference element using the relation
\begin{align*}
\int_e f(\vx) \;dx &= \int_{\hat e} (f\circ\phi_e)(\hat\vx)|\ten{J}_e(\hat\vx)| \;d\hat x,
\end{align*}
where $\ten{J}_e:=\hat\nabla\phi_e$ denotes the Jacobian of the element mapping, $\hat\nabla$ is the gradient operator of the reference space, and $|\ten{J}_e|:=|\mathrm{det}(\ten{J}_e)|$. 

A major drawback of the classical finite element method for the wave equation is that the mass matrix is not strictly diagonal. Since time stepping requires  computing terms of the form $L_h\vvu = M^{-1}A\vvu$ at each time step, a non-diagonal mass matrix would require solving a large system of equations at each time step. In practice, the mass matrix is therefore lumped into a diagonal matrix, so that the system of equations becomes trivial to solve. A lumped mass matrix can be obtained by placing the nodes of the basis functions $\hat{\mathcal{Q}}_I$ at the quadrature points for the mass matrix, so by setting $\hat{\mathcal{Q}}_I=\hat{\mathcal{Q}}_M$. To obtain 1D mass-lumped elements, we can use Gauss-Lobatto points. This can be extended to quadrilateral and hexahedral mass-lumped elements by using tensor-product basis functions. The resulting scheme is known as the spectral element method \cite{patera84,seriani94,komatitsch98}. For linear triangular and tetrahedral elements, we can place the nodes at the vertices and use a Newton--Cotes integration rule. For higher-order triangular and tetrahedral elements, however, we need to enrich the element space with higher-degree bubble functions and use special quadrature rules to maintain stability and accuracy after mass-lumping \cite{cohen95,cohen01,mulder96,chin99,mulder13,liu17,cui17,geevers18b}. 

\begin{figure}[h]
\centering
\begin{subfigure}[b]{0.3\textwidth}
  \includegraphics[width=\textwidth]{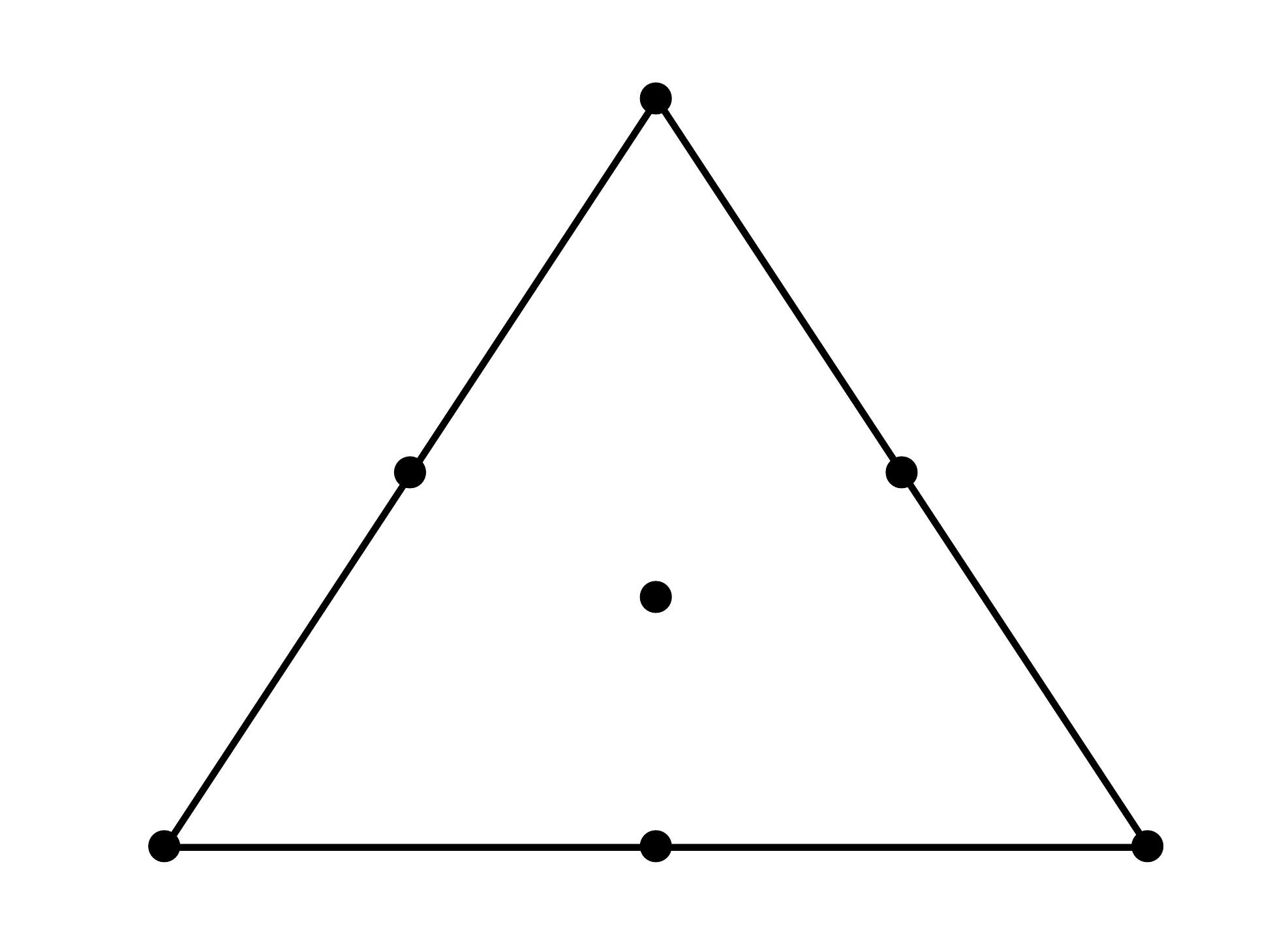}
\end{subfigure} \,
\begin{subfigure}[b]{0.3\textwidth}
  \includegraphics[width=\textwidth]{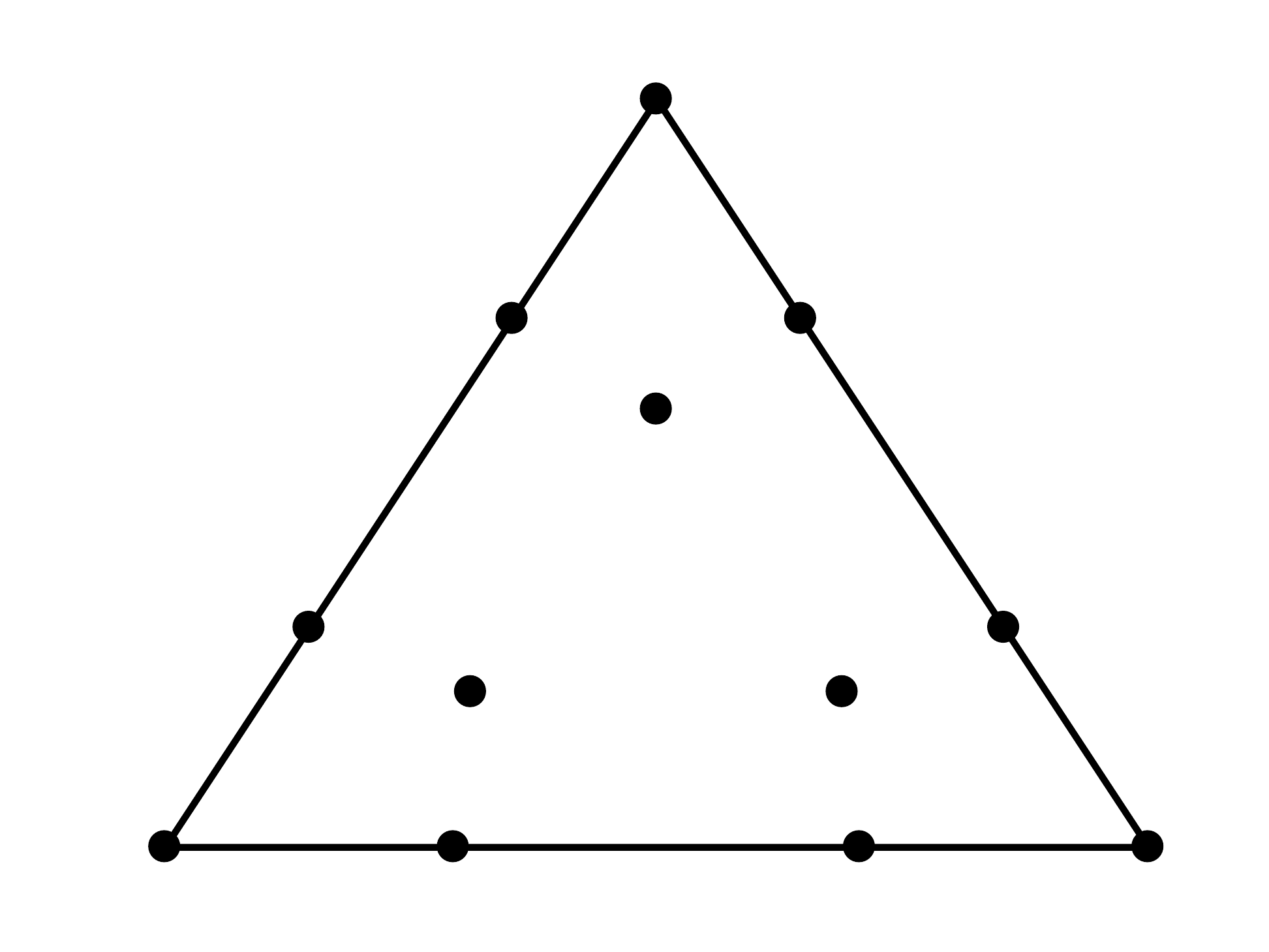}
\end{subfigure}
\caption{Mass-lumped degree-2 (left) and degree-3 (right) triangular element.}
\label{fig:MLTRI}
\end{figure}

In this paper, we test using the 1D spectral elements, the linear mass-lumped triangular element, and the quadratic and cubic mass-lumped triangular element presented in \cite{cohen95}. The element space of the quadratic mass-lumped triangular element is given by
\begin{align*}
\hat U &:= \mathcal{P}_2(\hat e) \oplus \{b\}.
\end{align*}
where $b:=\hat{x}_1\hat{x}_2\hat{x}_3$ is the degree-3 bubble function, with $\hat{x}_i$ the barycentric coordinates. The nodes of this element are placed at the 3 vertices, the midpoint of the 3 edges, and the centre of the triangle. The space of the cubic mass-lumped triangular element is given by
\begin{align*}
\hat U &:= \mathcal{P}_3(\hat e) \oplus \big(\{b\}\otimes\mathcal{P}_1(\hat e)\big).
\end{align*}
The nodes of this element are placed at the three vertices, the 6 points on the edges with barycentric coordinates $\alpha$, $1-\alpha$, and $0$, and at the three interior points with barycentric coordinates $\beta$, $\beta$, and $1-2\beta$, where
\begin{align*}
\alpha \approx  0.2934695559090402, \qquad
\beta   \approx  0.2073451756635909.
\end{align*}
An illustration of these elements is given in Figure \ref{fig:MLTRI}.

\subsection{Test with sharp contrasts in material parameters}
First, we test the finite element method with and without pre- and post-processing on a 1D periodic domain with sharp contrasts in material parameters and using a mesh with sharp contrasts in the element size. The test is similar to those in \cite{ryan05}, except that, unlike in \cite{ryan05}, super-convergence is observed on the entire domain, also near the material interfaces and sharp contrasts in the element size.

Let $\Omega:=(0,5)$ be the periodic domain and let the parameters $\rho$ and $c$ be given by
\begin{align*}
\rho(x) = \begin{cases}
1, & x\in(0,1), \\
\frac14, & x\in(1,5), \\ 
\end{cases}
\qquad
c(x) = \begin{cases}
1, & x\in(0,1), \\
4, & x\in(1,5). \\ 
\end{cases}
\end{align*} 
The exact solution is given by the travelling wave $u(x,t):=\sin(2\pi(X(x)-t))$, where
\begin{align}
\label{eq:defX}
X(x) &= \begin{cases}
x, & x\in(0,1),\\
\frac14x+\frac34, & x\in(1,5).\\
\end{cases}
\end{align}%
The simulated time interval is $(0,T)$ with $T=10$ and the initial conditions are obtained from the exact solution. 

\begin{figure}[h]
\includegraphics[width=\textwidth]{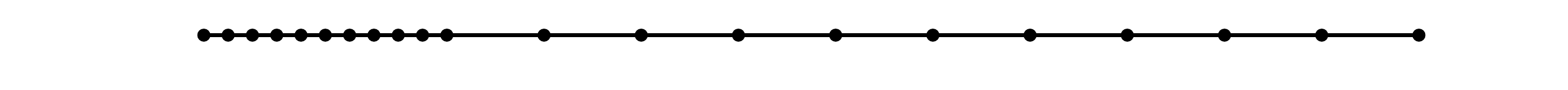}
\caption{Mesh with $N=10$. Dots represent the vertices.}
\label{fig:1Dmesh}
\end{figure}

We solve the wave equation numerically using degree-$p$ spectral elements for the time-stepping scheme and degree-$2p$ spectral elements for the order-$q$ improving pre- and post-processing steps. In case $q=0$, no pre- and post-processing is applied. We use $N$ elements per wavelength, so $\Delta x = \frac{1}{N}$ in $(0,1)$ and $\Delta x = \frac{4}{N}$ in $(1,5)$ and therefore we have a sharp contrast in mesh size at $x=0$ and $x=1$. An illustration of the mesh is given in Figure \ref{fig:1Dmesh}

The relative errors in the weighted energy norm and weighted $L^2$-norm are defined as follows:
\begin{align*}
e_0 &:= \frac{\big\| \rho^{1/2}\big(u(T)-u_{h*}^{N_T}\big)\big\|_0}{\|\rho^{1/2}u(T)\|_0}, \\ 
e_E &:= \frac{\big\|\rho^{1/2}\big(\partial_tu(T)-v_{h*}^{N_T}\big)\big\|_0 + \big\|c^{1/2}\nabla\big(u(T)-u_{h*}^{N_T}\big)\big\|_0}{\|\rho^{1/2}\partial_tu(T)\|_0 + \|c^{1/2}\nabla u(T)\|_0}.
\end{align*}

\begin{table}[h]
\caption{Results for the 1D problem showing relative errors in the energy norm of the degree-$p$ spectral element method with ($q=p$) and without ($q=0$) order-$q$ improving pre- and post-processing and using $N$ elements per wavelength.}
\label{tab:err1Dp}
\begin{center}
\begin{tabular}{c r |r r r |r r r}
		&		& \multicolumn{3}{c}{$q=0$} 	& \multicolumn{3}{|c}{$q=p$} 	\\ 
		& N 		& $e_{E}$ 	& ratio	& order 	& $e_{E}$ 	& ratio	& order	\\ \hline
		& 20		& 1.66e-01	& 		& 		& 1.54e-01	& 		&  		\\  
$p=1$	& 40		& 4.90e-02	& 3.39	& 1.76	& 3.86e-02	& 4.00	& 2.00	\\ 
		& 80		& 1.71e-02	& 2.86	& 1.52	& 9.65e-03	& 4.00	& 2.00	\\ \hline
		&5		& 7.23e-02	& 		&		& 6.02e-02	& 		&		\\ 
$p=2$	&10		& 9.68e-03	& 7.47	& 2.90	& 3.63e-03	& 16.6	& 4.05	\\  
		& 20		& 2.00e-03	& 4.85	& 2.28	& 2.25e-04	& 16.1	& 4.01	\\ \hline
		& 5 		& 3.56e-03	& 		&		& 4.12e-04	& 		&		\\ 
$p=3$	& 10  	& 4.19e-04	& 8.50	& 3.09	& 6.41e-06	& 64.2	& 6.01	\\  
		& 20		& 5.04e-05	& 8.31	& 3.05	& 1.00e-07	& 64.0	& 6.00		
\end{tabular}
\end{center}
\end{table}

\begin{table}[h]
\caption{Results for the 1D problem showing relative errors in the $L^2$- and energy norm of the degree-3 spectral element method with order-$q$ improving pre- and post-processing and using $N$ elements per wavelength.}
\label{tab:err1Dq}
\begin{center}
\begin{tabular}{c r |r r r |r r r}
		&  		& \multicolumn{6}{c}{$p=3$} 	\\ 
		& N		& $e_{0}$ 		& ratio	& order 	& $e_{E}$ 	& ratio	& order	\\ \hline
		& 5		& 4.14e-04	& 		&		& 6.58e-04	& 		&		\\ 
$q=1$	&10  		& 6.42e-06	& 64.4	& 6.01	& 3.12e-05	& 21.1	& 4.40	\\  
		& 20		& 1.00e-07	& 64.0	& 6.00	& 1.70e-06	& 18.4	& 4.20	\\ \hline
		& 5		& 4.12e-04	& 		&		& 4.25e-04	& 		&		\\ 
$q=2$	& 10  	& 6.43e-06	& 64.1	& 6.00	& 7.35e-06	& 57.7	& 5.85	\\  
		& 20		& 1.00e-07	& 64.1 	& 6.00	& 1.41e-07	& 52.1	& 5.70		
\end{tabular}
\end{center}
\end{table}

Results for different elements and different pre/post-processing schemes are given in Tables \ref{tab:err1Dp} and \ref{tab:err1Dq}. The errors are computed using the quadrature rule of the degree-$2p$ spectral element method. The results confirm that the convergence rate is of order $q+p$ in the energy norm and $q+p+1$ in the $L^2$-norm. In case $p=3$ and $q=1$, this scheme even seems to converge with order 6 in the $L^2$-norm. However, this higher convergence rate only appears in the 1D case and not in the 2D case as we will show in the next subsection.

\begin{figure}[h]
\begin{center}
\includegraphics[width=0.5\textwidth]{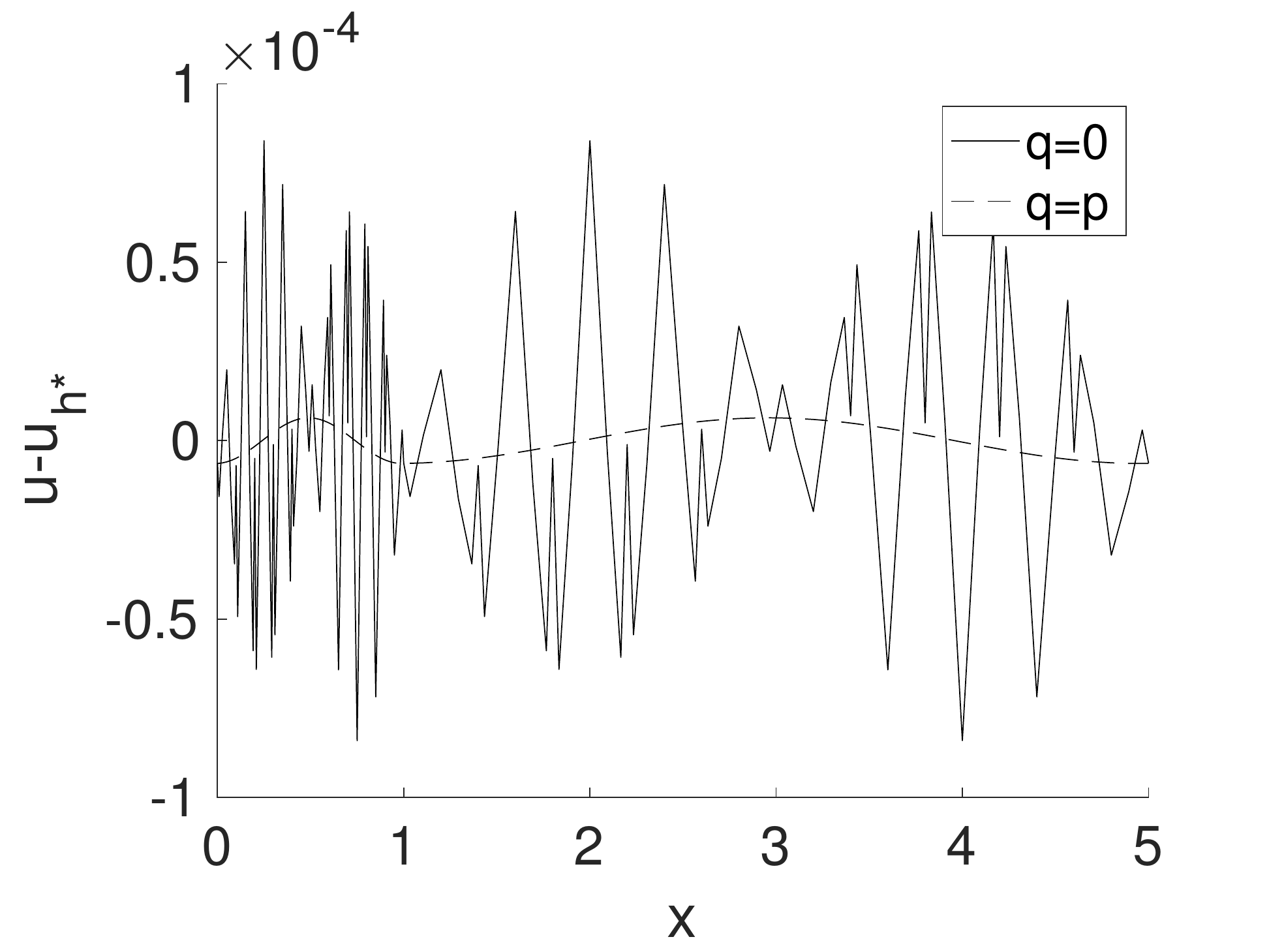}
\caption{Error $u-u_{h*}$ at final time $T$ of the unprocessed (solid) and post-processed (dashed, $q=p$) finite element solution for the 1D problem using $N=10$ and $p=3$.}
\label{fig:err1D}
\end{center}
\end{figure}

Figure \ref{fig:err1D} also shows the error of the unprocessed and post-processed finite element approximation for $N=10$ and $p=3$. The figure illustrates that the error of the unprocessed approximation is highly oscillatory, while the error of the post-processed solution is much smoother.

Since there are strong contrasts in the material parameters, it is not immediately obvious that the regularity assumption given in \eqref{eq:reg} holds. However, by using the piecewise-linear coordinate transformation $x\rightarrow X$, given in \eqref{eq:defX}, the problem can be rewritten as a wave propagation problem for a periodic domain with constant material parameters.

\subsection{Test using unstructured triangular meshes}
\label{sec:testSquare}

\begin{figure}[h]
\centering
\begin{subfigure}[b]{0.3\textwidth}
  \includegraphics[width=\textwidth]{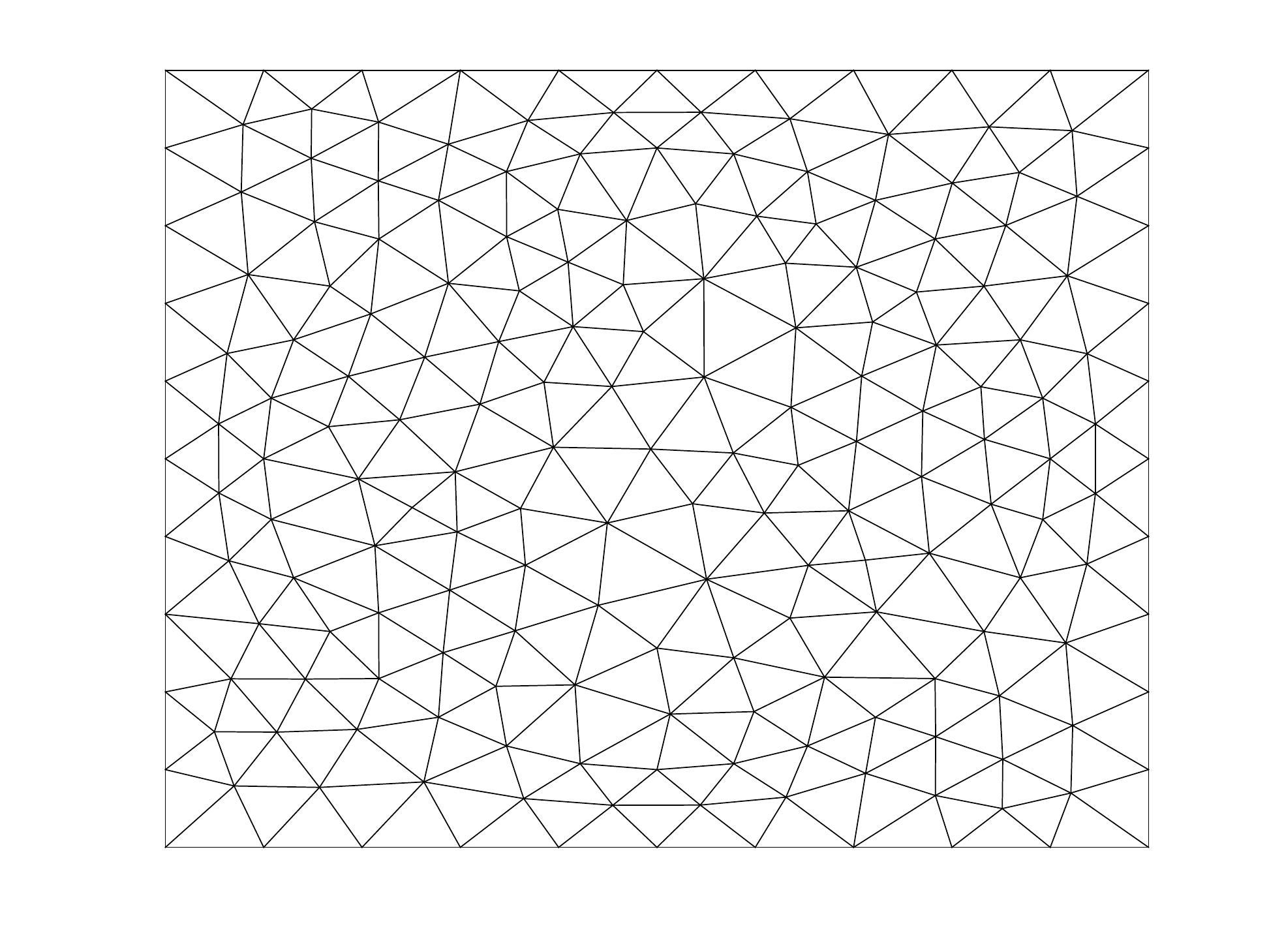}
  \caption*{mesh 1}
  \label{fig:meshc1}
\end{subfigure} \,
\begin{subfigure}[b]{0.3\textwidth}
  \includegraphics[width=\textwidth]{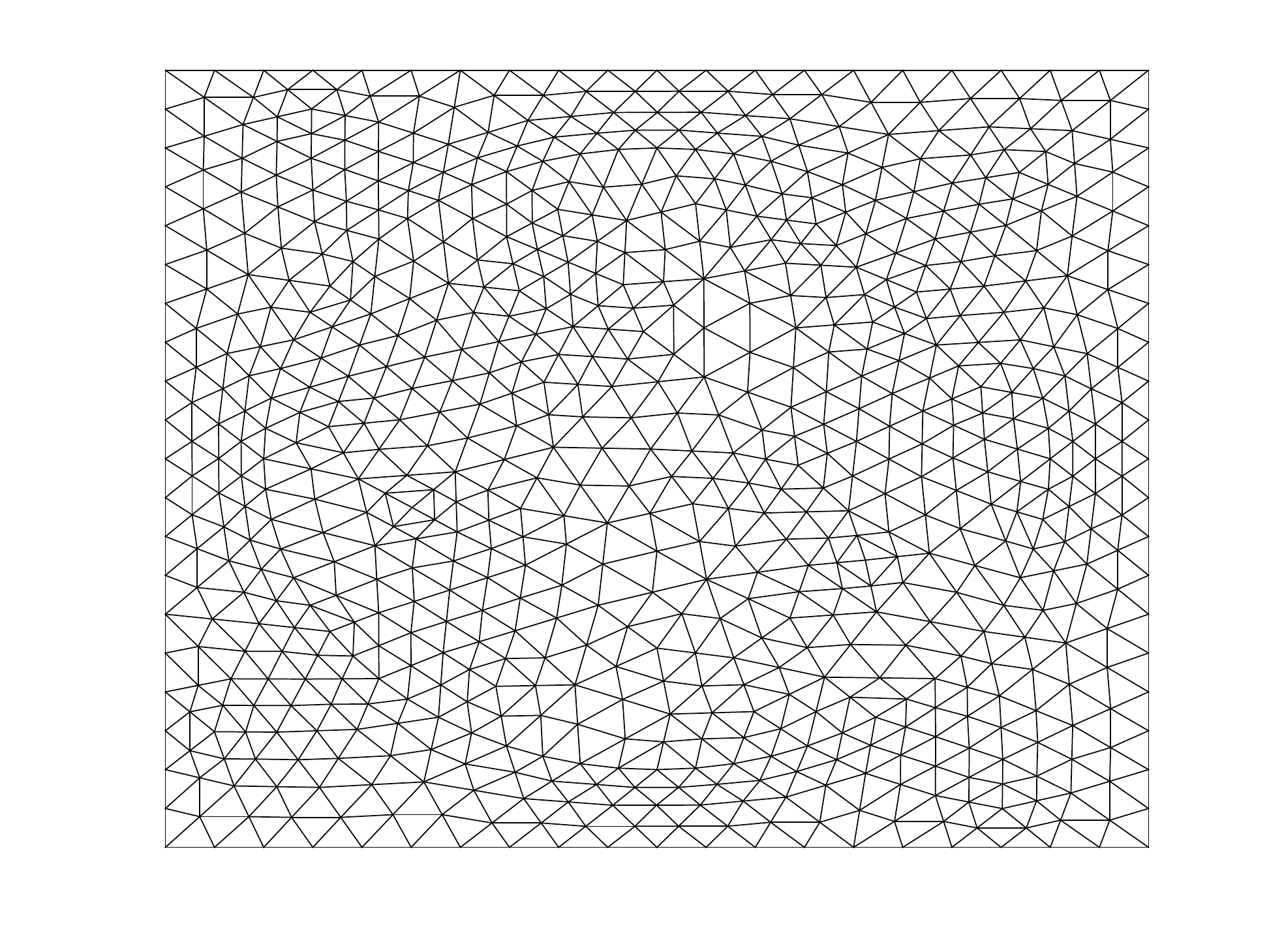}
  \caption*{mesh 2}
  \label{fig:meshc2}
\end{subfigure} \,
\begin{subfigure}[b]{0.3\textwidth}
  \includegraphics[width=\textwidth]{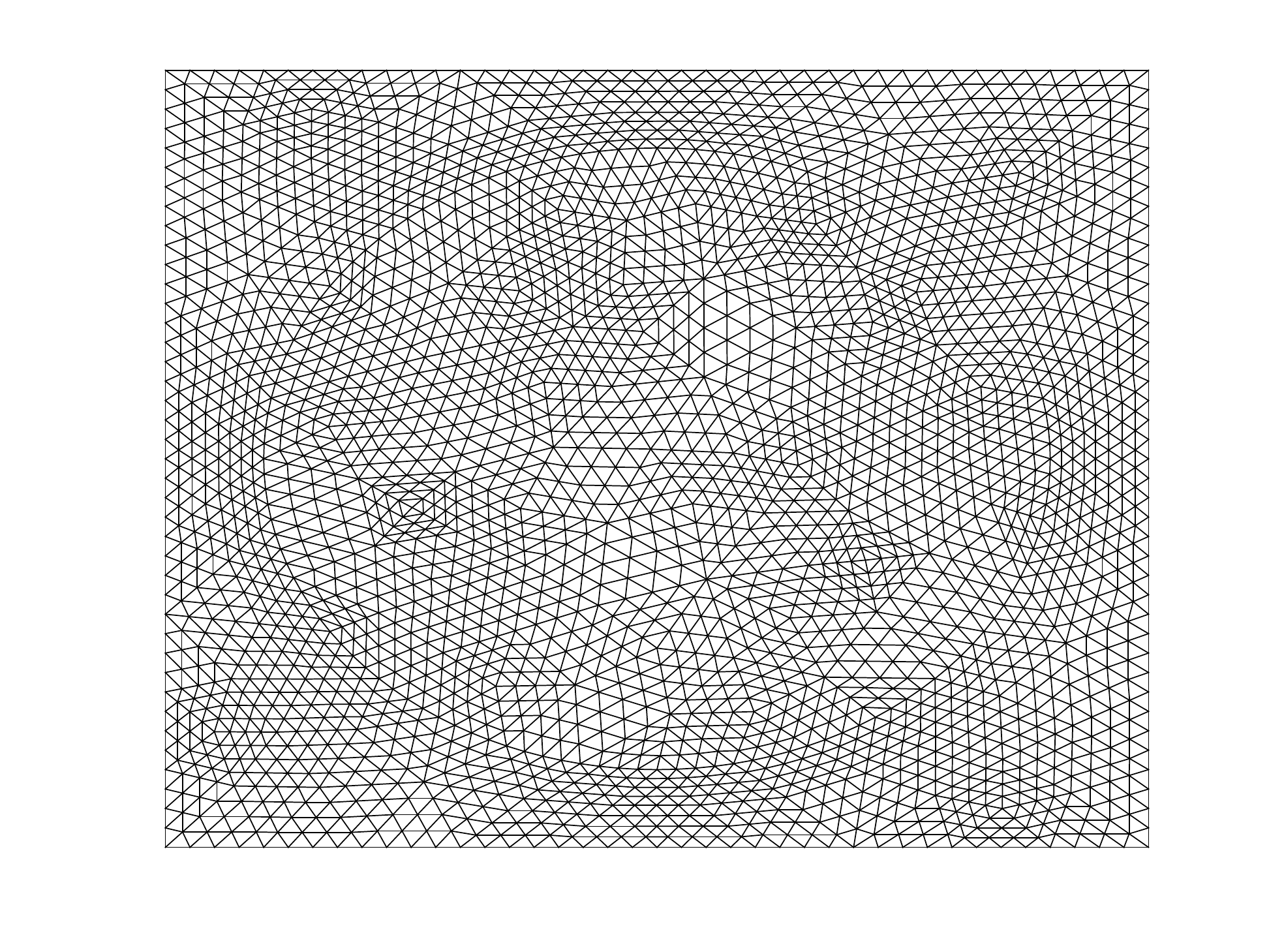}
  \caption*{mesh 3}
  \label{fig:meshc3}
\end{subfigure}
\caption{Three unstructured triangular meshes of a square domain, where mesh 2 is a refinement of mesh 1 and mesh 3 is a refinement of mesh 2.}
\label{fig:meshc}
\end{figure}

Next, we test the pre- and post-processing methods on a squared domain with non-constant material parameters and zero Dirichlet boundary conditions using unstructured triangular meshes. The test is similar to those in \cite{mirzaee13}, except that, unlike in \cite{mirzaee13}, super-convergence is still clearly observed when using unstructured meshes. 

Let $\Omega:=[0,1]^2$ be the spatial domain and $(0,T)$ the time domain with final time $T=50$. Introduce distortions of the Cartesian coordinates, given by
\begin{align}
\label{eq:defXY}
X(x) := x + \frac{1}{10}\sin(\pi x), \qquad Y(y) := y + \frac{1}{10}\sin(\pi y),
\end{align}
and define the spatial parameters
\begin{align*}
\rho(x,y) := \frac{\dot{X}^2(x) + \dot{Y}^2(y)}{2\dot{X}(x)\dot{Y}(y)}, \qquad c(x,y):=\frac{1}{\dot{X}(x)\dot{Y}(y)},
\end{align*}
where $\dot{X}$ and $\dot{Y}$ denote the derivatives of $X$ and $Y$, respectively. The wave propagation speed $\sqrt{c/\rho}$ then takes values between $(1+\frac{1}{10}\pi)^{-1}\approx 0.76$ and $(1-\frac{1}{10}\pi)^{-1}\approx 1.46$. The exact solution is given by
\begin{align*}
u(x,y,t) &:= \sin(2\pi X(x))\sin(2\pi Y(y))\cos(2\pi t + \phi),
\end{align*}
with source term
\begin{align*}
f(x,y,t) &:= 4\pi^2\sin(2\pi X(x))\sin(2\pi Y(y))\cos(2\pi t + \phi),
\end{align*}
where $\phi:=0.5$ is an arbitrary phase shift added to avoid the initial displacement or velocity field to be completely zero. The initial conditions are obtained from the exact solution.

\begin{table}[h]
\caption{Results for the square-domain problem showing relative errors in the energy norm of the degree-$p$ mass-lumped triangular element method with ($q=p$) and without ($q=0$) order-$q$ improving pre- and post-processing.}
\label{tab:err2Dp}
\begin{center}
\begin{tabular}{c r |r r r |r r r}
		&		& \multicolumn{3}{c}{$q=0$} 	& \multicolumn{3}{|c}{$q=p$} 	\\ 
		& mesh 		& $e_{E}$ 	& ratio	& order 	& $e_{E}$ 	& ratio	& order	\\ \hline
		& 3		& 4.02e-02	& 		& 		& 4.24e-03	& 		&  		\\  
$p=1$	& 4		& 1.92e-02	& 2.09	& 1.07	& 1.09e-03	& 3.88	& 1.96	\\ 
		& 5		& 9.33e-03	& 2.06	& 1.04	& 2.70e-04	& 4.05	& 2.02	\\ \hline
		& 1		& 1.52e-02	& 		&		& 4.64e-04	& 		&		\\ 
$p=2$	& 2		& 3.80e-03	& 4.01	& 2.01 	& 2.89e-05	& 16.1	& 4.01	\\  
		& 3		& 9.46e-04	& 4.01	& 2.00	& 1.82e-06	& 15.9	& 3.99	\\ \hline
		& 1 		& 1.00e-03	& 		&		& 1.16e-06	& 		&		\\ 
$p=3$	& 2	  	& 1.26e-04	& 7.92	& 2.99	& 1.75e-08	& 66.1	& 6.05	\\  
		& 3		& 1.58e-05	& 8.00	& 3.00	& 2.97e-10	& 59.0	& 5.88					
\end{tabular}
\end{center}
\end{table}

\begin{table}[h]
\caption{Results for the square-domain problem showing relative errors in the $L^2$- and energy norm of the degree-3 mass-lumped triangular element method with order-$q$ improving pre- and post-processing.}
\label{tab:err2Dq}
\begin{center}
\begin{tabular}{c r |r r r |r r r}
		&  		& \multicolumn{6}{c}{$p=3$} 	\\ 
		& mesh		& $e_{0}$ 		& ratio	& order 	& $e_{E}$ 	& ratio	& order	\\ \hline
		& 1		& 1.68e-06	& 		&		& 3.24e-05	& 		&		\\ 
$q=1$	& 2  		& 4.06e-08	& 41.4	& 5.37	& 1.99e-06	& 16.3	& 4.03	\\  
		& 3		& 1.11e-09	& 36.5	& 5.19	& 1.22e-07	& 16.3	& 4.02	\\ \hline
		& 1		& 1.27e-06	& 		&		& 5.09e-06	& 		&		\\ 
$q=2$	& 2  		& 1.92e-08	& 65.8	& 6.04	& 1.51e-07	& 33.7	& 5.07	\\  
		& 3		& 3.23e-10	& 59.6 	& 5.90	& 4.56e-09	& 33.2	& 5.05		
\end{tabular}
\end{center}
\end{table}

We consider 5 different meshes where each subsequent mesh has a resolution twice as fine as the previous mesh. An illustration of the first three meshes is given in Figure \ref{fig:meshc}. We test the standard linear mass-lumped triangular element and the quadratic and cubic mass-lumped triangular elements presented in \cite{cohen95}. For post-processing, we use standard degree-$2p$ Lagrangian elements combined with a degree-$4p$ accurate quadrature rule taken from \cite{zhang09} for evaluating the integrals. This last quadrature rule is also used to evaluate the errors.

\begin{figure}[h]
\centering
\begin{subfigure}[b]{0.4\textwidth}
  \includegraphics[width=\textwidth]{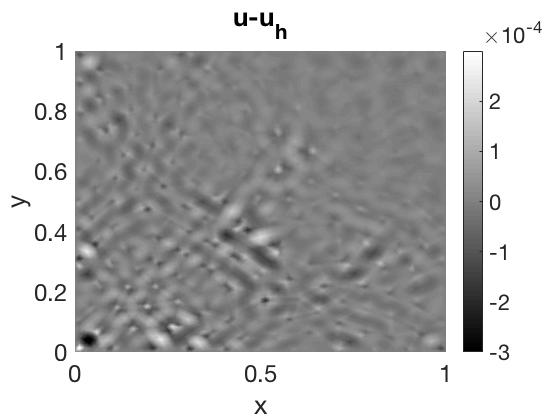}
\end{subfigure} \,
\begin{subfigure}[b]{0.4\textwidth}
  \includegraphics[width=\textwidth]{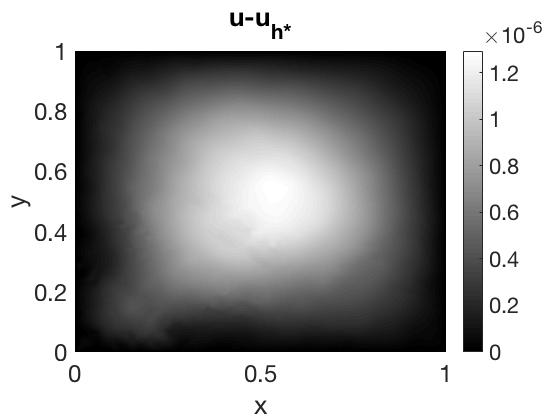}
\end{subfigure}
\caption{Error $u-u_{h*}$ at final time $T$ of the unprocessed (left) and post-processed finite element solution (right, $q=p$) for the square-domain problem using mesh 1 and $p=3$.}
\label{fig:err2D}
\end{figure}

Results for different elements and different pre/post-processing schemes are given in Tables \ref{tab:err2Dp} and \ref{tab:err2Dq}. In all cases, the convergence rate is of order $q+p$ in the energy norm and $q+p+1$ in the $L^2$-norm. Figure \ref{fig:err2D} also shows the error of the unprocessed and post-processed finite element approximation for mesh 1 and $p=3$. The figure illustrates again that the error of the unprocessed approximation is highly oscillatory, while the error of the post-processed solution is relatively smooth.

Since the boundary of the squared domain is not smooth, it is again not immediately obvious that the regularity assumption given in \eqref{eq:reg} holds. However, by using the smooth coordinate transformations $(x,y)\rightarrow (X,Y)$, given in \eqref{eq:defXY}, the problem can be rewritten as a wave propagation problem on a unit square with constant material parameters. The regularity assumption \eqref{eq:reg} then follows from Fourier analysis.

\subsection{Test on a domain with a curved boundary}
\label{sec:testCircle}
Finally, we test the pre- and post-processing algorithms on a circular domain with zero Dirichlet boundary conditions. Let $\Omega$ be the unit circle , let $(0,T)$ be the time domain with $T=50$, and let $\rho=c=1$ be constant. The exact solution is given by
\begin{align*}
u(r,\theta,t) &:= J_2(\kappa r)\cos(2\theta)\cos(\kappa t + \phi),
\end{align*}
where $r\in[0,1)$ and $\theta\in[0,2\pi)$ are the polar coordinates, $\phi:=0.5$ is an arbitrary phase shift, $J_2$ is the Bessel function of order 2, and $\kappa\approx 5.135622301840683$ is the first positive root of $J_2$. The initial conditions are obtained from the exact solution.

\begin{figure}[h]
\centering
\begin{subfigure}[b]{0.3\textwidth}
  \includegraphics[width=\textwidth]{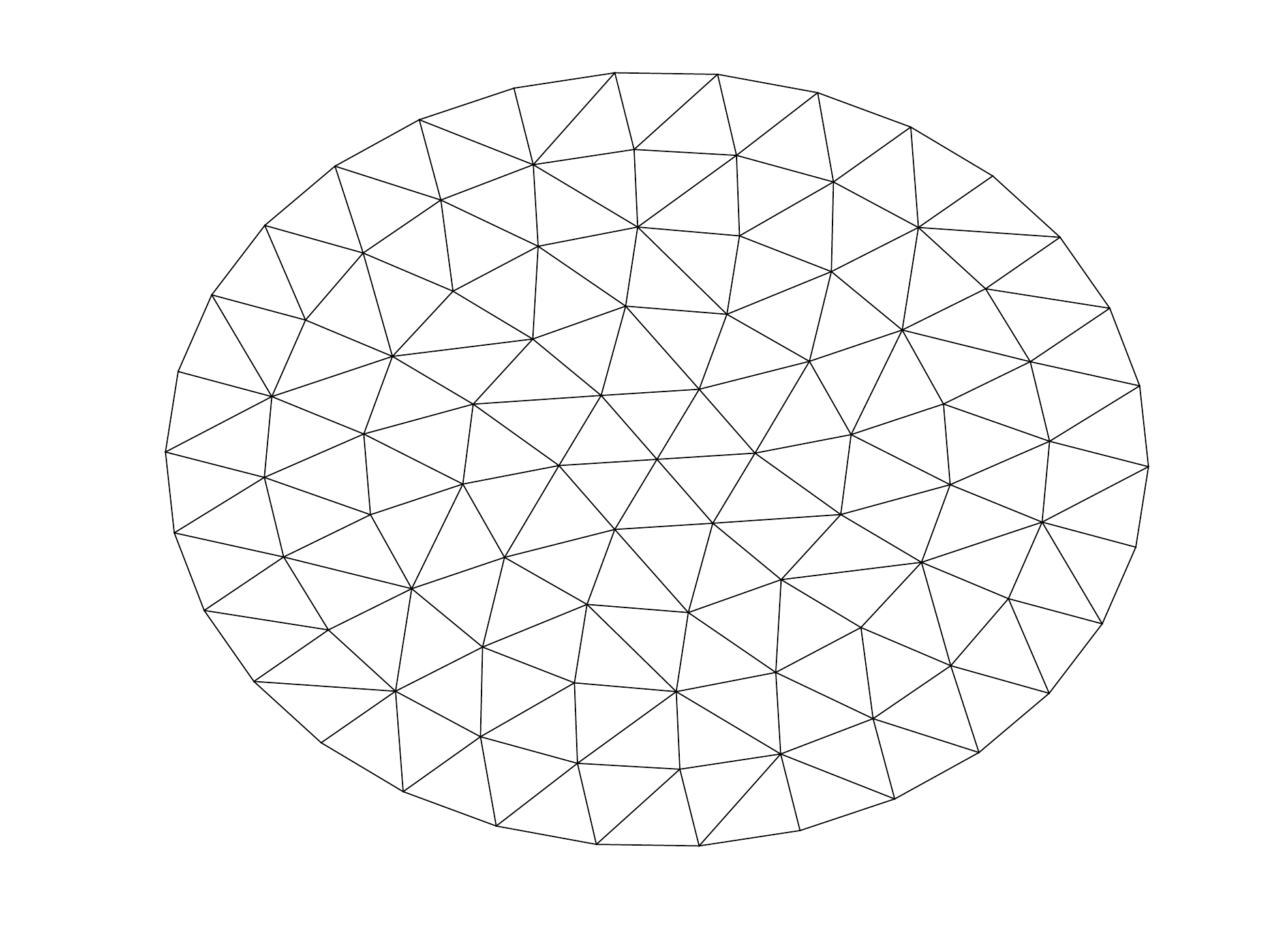}
  \caption*{mesh 1}
  \label{fig:mesh1}
\end{subfigure} \,
\begin{subfigure}[b]{0.3\textwidth}
  \includegraphics[width=\textwidth]{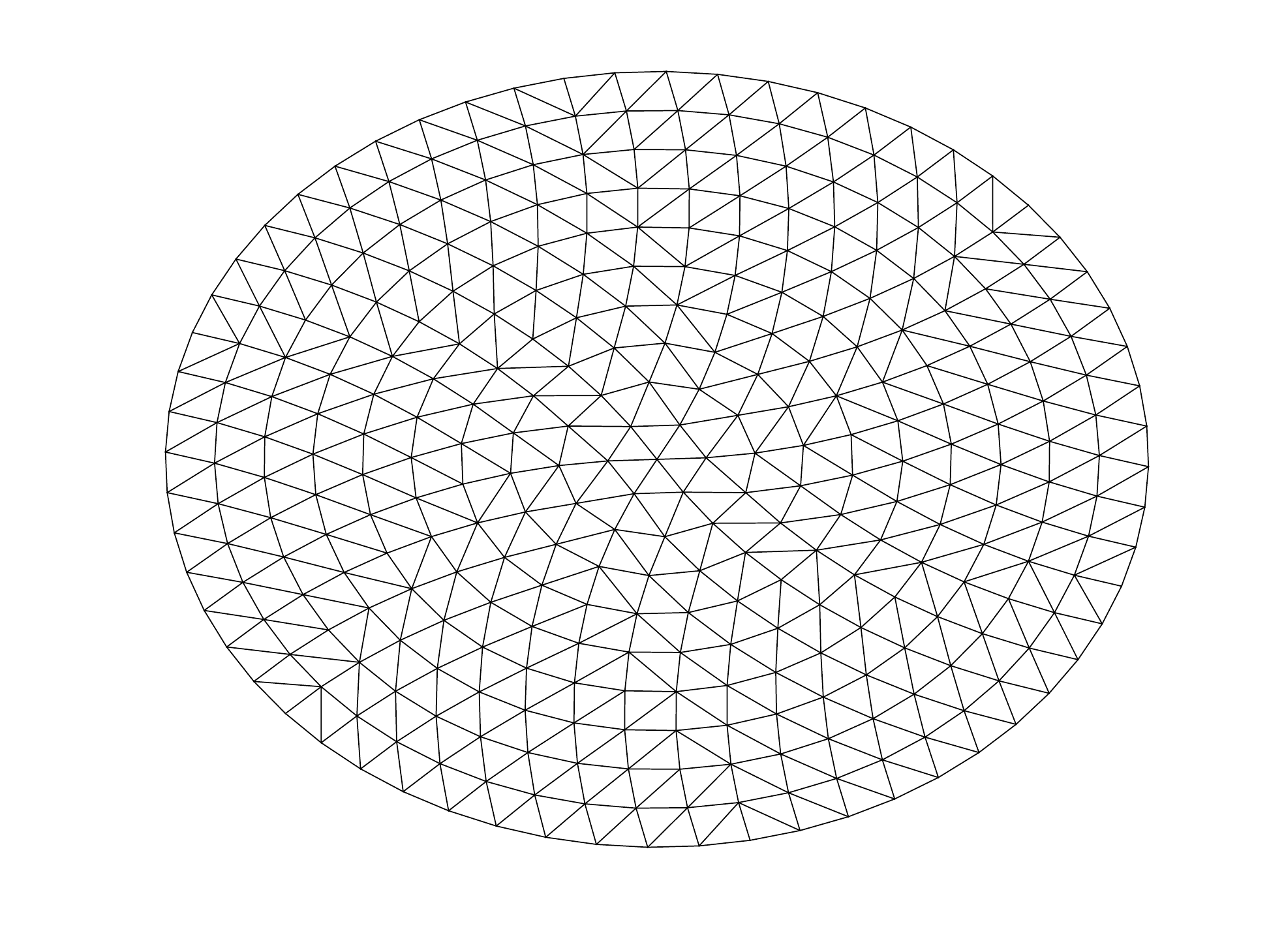}
  \caption*{mesh 2}
  \label{fig:mesh2}
\end{subfigure} \,
\begin{subfigure}[b]{0.3\textwidth}
  \includegraphics[width=\textwidth]{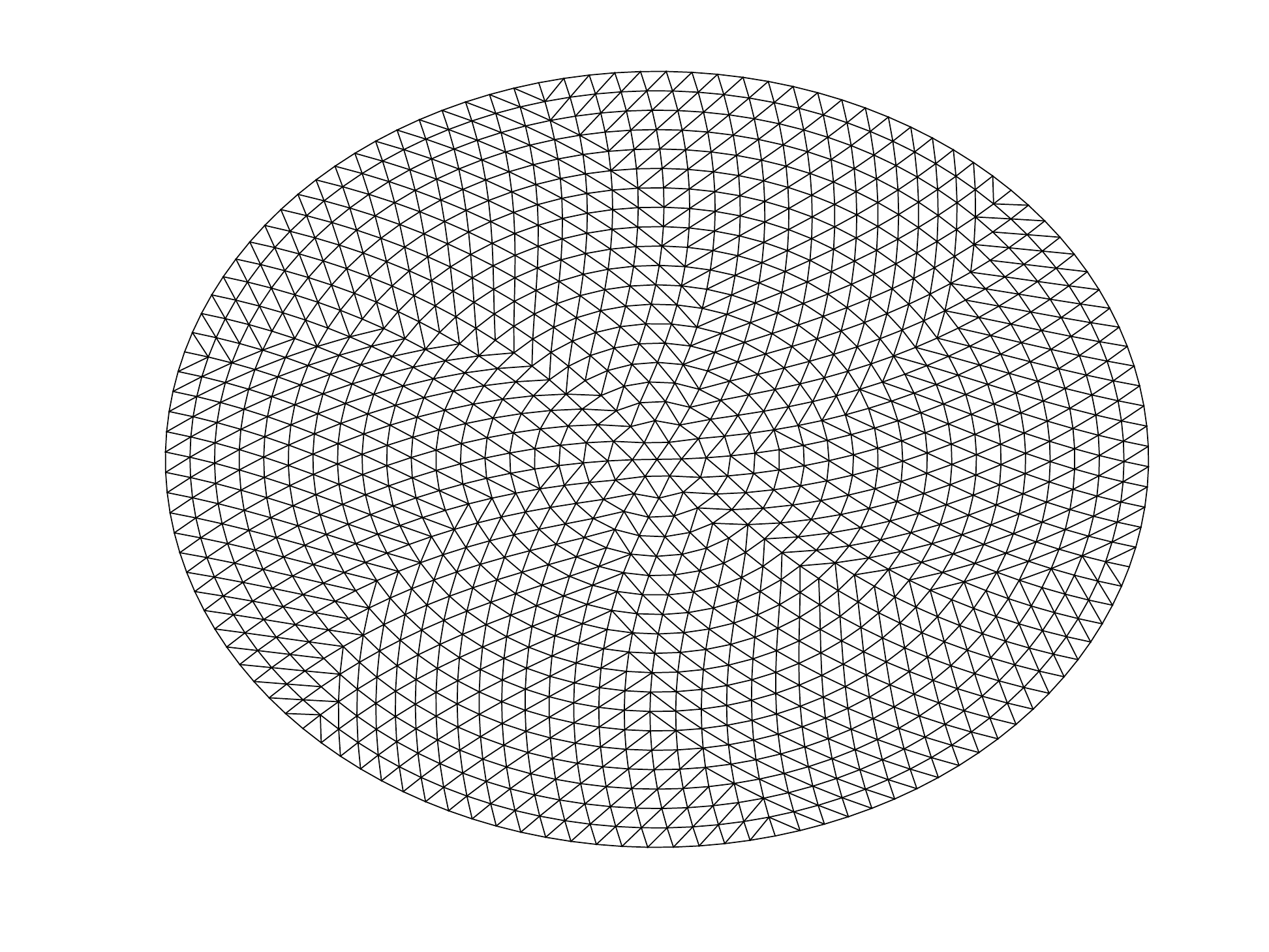}
  \caption*{mesh 3}
  \label{fig:mesh3}
\end{subfigure}
\caption{Three unstructured triangular meshes of a circular domain, where mesh 2 is a refinement of mesh 1 and mesh 3 is a refinement of mesh 2.}
\label{fig:mesh}
\end{figure}

\begin{table}[h]
\caption{Results for the circle-domain problem showing relative errors in the energy norm of the degree-$p$ mass-lumped triangular element method with ($q=p$) and without ($q=0$) order-$q$ improving pre- and post-processing.}
\label{tab:err2Dcp}
\begin{center}
\begin{tabular}{c r |r r r |r r r}
		&		& \multicolumn{3}{c}{$q=0$} 	& \multicolumn{3}{|c}{$q=p$} 	\\ 
		& mesh 	& $e_{E}$ 	& ratio	& order 	& $e_{E}$ 	& ratio	& order	\\ \hline
		& 3		& 4.21e-01	& 		& 		& 4.20e-01	& 		&  		\\  
$p=1$	& 4		& 1.09e-01	& 3.85	& 1.95	& 1.00e-01	& 4.20	& 2.07	\\ 
		& 5		& 3.28e-02	& 3.32	& 1.73	& 2.45e-02	& 4.07	& 2.03	\\ \hline
		& 1		& 4.40e-02	& 		&		& 2.28e-02	& 		&		\\ 
$p=2$	& 2		& 7.90e-03	& 5.57	& 2.48	& 1.27e-03	& 17.9	& 4.16	\\  
		& 3		& 1.87e-03	& 4.23	& 2.08	& 7.35e-05	& 17.3	& 4.11	\\ \hline
		& 1 		& 2.68e-03	& 		&		& 1.75e-04	& 		&		\\ 
$p=3$	& 2	  	& 3.35e-04	& 8.01	& 3.00	& 2.22e-06	& 78.8	& 6.30	\\  
		& 3		& 4.20e-05	& 7.98	& 3.00	& 2.96e-08	& 75.0	& 6.23				
\end{tabular}
\end{center}
\end{table}

We test again on 5 different meshes where each subsequent mesh has a resolution twice as fine as the previous mesh. An illustration of the first three meshes is given in Figure \ref{fig:mesh}. To obtain a higher accuracy, we replace the straight elements at the boundary by curved elements. To parametrise these curved elements, we use degree-$2p$ hyperparametric element mappings. First, we place $2p+1$ nodes at equal intervals on every edge of a boundary element, with the outer nodes lying on the vertices. We do something similar for the edges of the reference element.  We then rescale the coordinates of the nodes on boundary edges such that they lie on the boundary of the unit circle. Let $\{\hat{\vx}_i\}_{i=1}^{6p}$ denote the nodes at the boundary of the reference element, and let $\{\vx_i\}_{i=1}^{6p}$ denote the corresponding nodes at the boundary of a curved element. We then parameterise the curved element with the element mapping $\phi_e$ defined such that $\phi_e(\hat{\vx}_i)=\vx_i$ for $i=1,\dots,6p$. To obtain this element mapping, we interpolate using the following set of polynomials of degree $2p$ and less:
\begin{align*}
&\hat{x}_i, &&i\in\{1,2,3\}, \\
&\hat{x}_i\hat{x}_j(\hat{x}_i-\hat{x_j})^k, &&i,j\in\{1,2,3\}, i<j, 0\leq k\leq 2p-2,
\end{align*}
where $\hat{x}_1, \hat{x}_2, \hat{x}_3$ denote the barycentric coordinates of the reference element. 

For the time-stepping and pre- and post-processing, we use the same elements and quadrature rules as in the previous subsection. The results for different elements and different pre/post-processing schemes are given in Table \ref{tab:err2Dcp}. Again, the results confirm that using pre- and post-processing can result in a convergence rate of order $2p$ instead of order $p$ in the energy norm.

\subsection{Pre- and post-processing using an iterative method}
In the previous subsections, we applied pre- and post-processing in combination with a direct solver. In practice, the number of degrees of freedom might become too large to efficiently use a direct solver. We therefore also test the pre- and post-processing method using an iterative method. 

\begin{table}[h]
\caption{Results for the square-domain problem showing relative errors in the energy norm $e_E$ of the degree-$p$ mass-lumped triangular element method with order-$p$ improving pre- and post-processing using the conjugate gradient method with $N_{it}$ iterations. The number of time steps $ N_T$ is listed in the third column.}
\label{tab:err2Dpit}
\begin{center}
\begin{tabular}{c r r |r r r r r}
		& mesh 	& $N_{T}$		& $N_{it}=0$ 	& 10		& 100 	& 1000 	& $\infty$	\\ \hline
		& 3		&   6190	& 4.02e-02	& 5.79e-03	& 4.22e-03	& 4.24e-03	& 4.24e-03 	\\  
$p=1$	& 4		& 12342	& 1.92e-02	& 1.67e-03	& 1.09e-03	& 1.09e-03	& 1.09e-03	\\ 
		& 5		& 24727	& 9.33e-03	& 4.59e-04	& 2.85e-04	& 2.70e-04	& 2.70e-04	\\ \hline
		& 1		&   3586	& 1.52e-02	& 6.52e-04	& 4.62e-04	& 4.64e-04	& 4.64e-04	\\ 
$p=2$	& 2		&   7222 	& 3.80e-03	& 9.36e-05 	& 2.91e-05	& 2.89e-05	& 2.89e-05	\\  
		& 3		& 14529 	& 9.46e-04	& 1.40e-05	& 1.98e-06	& 1.82e-06	& 1.82e-06 	\\ \hline
		& 1 		&   7399 	& 1.00e-03	& 1.19e-04	& 2.91e-06	& 1.16e-06	& 1.16e-06	\\ 
$p=3$	& 2	  	& 14879 	& 1.25e-04	& 1.16e-05	& 4.55e-07	& 1.77e-08	& 1.75e-08 	\\  
		& 3		& 29827	& 1.58e-05	& 1.38e-06	& 3.92e-08	& 7.28e-10	& 2.97e-10					
\end{tabular}
\end{center}
\end{table}

We consider again the test on the heterogeneous squared domain using the unstructured triangular meshes of Section \ref{sec:testSquare}. For pre- and post-processing, we need to recursively compute $\vvct{\partial_t^k u}_{h}^0$ and $\vvct{\partial_t^k u}_{h*}^{N_T}$ using \eqref{eq:wave4b} and \eqref{eq:wave4c}, respectively, which requires solving linear systems of equations of the form $A\vvct{x} = \vvct{y}$ and $A_*\vvct{x}_* = \vvct{y}_*$. We solve these systems using the conjugate gradient method with $N_{it}$ iterations. As a preconditioner, we use a diagonal matrix computed by taking row sums of the absolute values of $A$ and $A_*$, and as an initial guess for $\vvct{\partial_t^ku}_{h}^0$ and $\vvct{\partial_t^ku}_{h*}^{N_T}$, we use $\vvct{\partial_t^ku|_{t=0}}$ and $P\vvct{\partial_t^ku}_{h}^{N_T}$, respectively.

The accuracy of the order-$p$ improving pre- and post-processing scheme using this iterative solver with different numbers of iterations $N_{it}$ is shown in Table \ref{tab:err2Dpit}, where $N_{it}=0$ means no pre- and post-processing is applied at all and $N_{it}=\infty$ means a direct solver was used. The table shows that $10$ iterations already reduces the error by an order of magnitude, $100$ iterations reduce the error by 2-3 orders of magnitude when using higher-degree elements, and $1000$ iterations reduce the error by 4-5 orders when using degree-3 elements on finer meshes.

\section{Conclusion}
\label{sec:conclusion}
We presented a new pre- and post-processing method to enhance the accuracy of the finite element approximation of linear wave problems. We proved that, by applying at most $q$ processing steps at only the initial and final time, the convergence rate of the finite element approximation can be improved from order $p+1$ to order $p+1+q$ in the $L^2$-norm, and from order $p$ to order $p+q$ in the energy norm, both up to a maximum of order $2p$. Each processing step corresponds to solving a linear system that can be solved efficiently with a direct solver or an iterative method and for the latter, good first guesses are provided by the unprocessed initial and final values. Numerical experiments showed that the pre- and post-processing steps can reduce the magnitude of the error by several orders, even when using an iterative method with only a small number of iterations. The experiments also confirmed that a convergence rate of order $2p$ in the energy norm is obtained, even when using unstructured meshes or meshes with sharp contrasts in element size and when solving wave propagation problems on domains with curved boundaries or strong heterogeneities.

\bibliographystyle{abbrv}
\bibliography{PPS}

\appendix
\section{Properties of the spatial operator}

\begin{lem}
\label{lem:L1}
The operators $\Lc^{-1}$ and $\Lh^{-1}$ satisfy the following symmetry properties: 
\begin{subequations}
\label{eq:L1}
\begin{align}
a(\Lc^{-1}u,v) &=  a(u,\Lc^{-1}v), &&u,v\in H^1_0(\Omega), \\
(\Lc^{-1}u,v)_{\rho}  &= (u,\Lc^{-1}v)_{\rho} , &&u,v\in L^2(\Omega), \\
a(\Lh^{-1}u,v) &=  a(u,\Lh^{-1}v), &&u,v\in U_h, \\
(\Lh^{-1}u,v)_{\rho} &= (u,\Lh^{-1}v)_{\rho}, &&u,v\in U_h.
\end{align}
\end{subequations}
\end{lem}

\begin{proof}
Using the definition of $\Lc^{-1}$, we can derive the first two equalities as follows:
\begin{align*}
a(\Lc^{-1}u,v) &= (u,v)_{\rho} = a(u,\Lc^{-1}v), \\
(\Lc^{-1}u,v)_{\rho}  &= a(\Lc^{-1}u,\Lc^{-1}v) = (u,\Lc^{-1}v)_{\rho}.
\end{align*}
The results for $\Lh^{-1}$ can be derived in an analogous way.
\end{proof}

\section{Finite element approximation properties}

\begin{lem}
\label{lem:int}
Let $p\geq 1$ denote the degree of the finite element space and let $u\in H^k(\Omega)$ for some $k\geq 0$. Then
\begin{subequations} 
\label{eq:int1}
\begin{align}
\|u-\Pih u\|_{0} &\leq Ch^{\min(p+1,k)}\|u\|_{\min(p+1,k)} &&\text{if } k\geq 0\label{eq:int1a} \\ 
\|u-\Pih u\|_{1} &\leq Ch^{\min(p,k-1)}\|u\|_{\min(p+1,k)} &&\text{if } k\geq 1. \label{eq:int1b}
\end{align}
\end{subequations}
\end{lem}

\begin{proof}
In \cite{bernardi89}, it is shown that these inequalities hold if we replace $\Pih$ by their quasi-interpolation operator $\Pi_h^{0}$, also when using curved elements. Inequality \eqref{eq:int1a} then follows from the fact that $\Pih$ minimises the approximation error in the weighted $L^2$-norm and inequality \eqref{eq:int1b} then follows from the inverse inequality. 
\end{proof}

\begin{rem}
The leading constants in Lemma \ref{lem:int} depend on the mesh regularity $\gamma$ as defined in \eqref{eq:defmeshreg}. For a sequence of meshes with curved elements, this regularity constant only remains bounded for an appropriate parameterisation of the curved elements, such as, for example, the parametrisation presented in \cite{lenoir86}.
\end{rem}

\begin{lem}
\label{lem:int2}
Assume regularity condition \eqref{eq:reg} holds for some $K\geq 0$ and let $u\in H^k(\Omega)$, with $k\geq 0$. Then, for all $\alpha\geq 0$, we have
\begin{subequations}
\label{eq:int2}
\begin{align}
\|\Lc^{-\alpha}(u-\Pih u)\|_0 &\leq Ch^{\min(p+1,k)+\min(p+1,2\alpha,K+2)}\|u\|_{\min(p+1,k)} &&\text{if }\alpha,k\geq 0,\label{eq:int2a} \\
\|\Lc^{-\alpha}(u-\Pih u)\|_1 &\leq Ch^{\min(p+1,k)+\min(p+1,2\alpha-1,K+2)}\|u\|_{\min(p+1,k)}  &&\text{if }\alpha+k\geq 1,\label{eq:int2b} 
\end{align}
\end{subequations}
where $p\geq 1$ denotes the degree of the finite element space.
\end{lem}

\begin{proof}
If $\alpha=0$, then the result follows immediately from Lemma \ref{lem:int}. Now, let $\alpha\geq 1$. Inequality \eqref{eq:int2a} immediately follows from the following:
\begin{align*}
\|\Lc^{-\alpha}(u-\Pih u)\|_0^2 &\leq C\big( \Lc^{-\alpha}(u-\Pih u),\Lc^{-\alpha}(u-\Pih u) \big)_{\rho} \\
&=  C\big( u-\Pih u, \Lc^{-2\alpha}(u-\Pih u) \big)_{\rho} \\
&=  C\big( u-\Pih u, (I-\Pih)\Lc^{-2\alpha}(u-\Pih u) \big)_{\rho} \\
&\leq  C\| u-\Pih u\|_0  \|(I-\Pih)\Lc^{-2\alpha}(u-\Pih u) \big\|_0 \\
&\leq  Ch^{\min(p+1,k)+\min(p+1,2\alpha,K+2)}\| u\|_{\min(p+1,k)} \|\Lc^{-\alpha}(u-\Pih u)\|_0,
\end{align*}
where $I$ denotes the identity operator and where the first line follows from the boundedness of $\rho$, the second line from Lemma \ref{lem:L1}, the third line from the definition of $\Pih$, the fourth line from the Cauchy--Schwarz inequality and the boundedness of $\rho$, and the last line from Lemma \ref{lem:int} and the regularity assumption. 

To prove \eqref{eq:int2b}, we start as follows:
\begin{align*}
\|\Lc^{-\alpha}(u-\Pih u)\|_1^2 &\leq Ca\big( \Lc^{-\alpha}(u-\Pih u),\Lc^{-\alpha}(u-\Pih u) \big) \\
&=  C\big( u-\Pih u, \Lc^{1-2\alpha}(u-\Pih u) \big)_{\rho},
\end{align*}
where the first line follows from the coercivity of $a$ and the second line from Lemma \ref{lem:L1} and the definition of $\Lc^{-1}$. For the rest of the proof, we can proceed in a way analogous to before.
\end{proof}

\begin{lem}
\label{lem:ell1}
Assume regularity condition \eqref{eq:reg} holds for some $K\geq 0$ and let $f\in H^k(\Omega)$, with $k\geq 0$. Then 
\begin{align}
\|(\Lc^{-1}-\Lh^{-1}\Pih) f\|_1 &\leq Ch^{\min(p,k+1,K+1)}\|f\|_{\min(p-1,k,K)},
\label{eq:ell1}
\end{align}
where $p\geq 1$ denotes the degree of the finite element space.
\end{lem}

\begin{proof}
This is a standard result of finite element approximations for elliptic problems. Define $u:=\Lc^{-1}f$ and $u_h:=\Lh^{-1}\Pih f$. From the definitions of $\Lc^{-1}$, $\Lh^{-1}$, and $\Pi_h$, it follows that 
\begin{align*}
a(u,w)&=(f,w) &&\text{for all }w\in H_0^1(\Omega), \\
a(u_h,w)&=(f,w) &&\text{for all }w\in U_h,
\end{align*}
and from the regularity assumption that $\|u\|_{\min(k+2,K+2)} \leq C\|f\|_{\min(k,K)}$. The lemma then follows from the coercivity and boundedness of $a$, Cea's Lemma, and Lemma \ref{lem:int}.
\end{proof}

\begin{lem}
\label{lem:ell2}
Assume regularity condition \eqref{eq:reg} holds for some $K\geq 0$ and let $f\in H^k(\Omega)$, with $k\geq 0$. Then 
\begin{subequations}
\begin{align}
\|\Lc^{-\alpha}(\Lc^{-1}-\Lh^{-1}\Pih) f\|_1 &\leq Ch^{\min(p,k+1,K+1)+\min(p,2\alpha,K+1)}\|f\|_{\min(p-1,k,K)}, \label{eq:ell2a} \\
\|\Lc^{-\alpha}(\Lc^{-1}-\Lh^{-1}\Pih) f\|_0 &\leq Ch^{\min(p,k+1,K+1)+\min(p,2\alpha+1,K+1)}\|f\|_{\min(p-1,k,K)}. \label{eq:ell2b}
\end{align}
\end{subequations}
for all $\alpha\geq 0$, where $p\geq 1$ denotes the degree of the finite element space.
\end{lem}

\begin{proof}
Define $u:=\Lc^{-1}f$ and $u_h:=\Lh^{-1}\Pih f$. We first prove \eqref{eq:ell2a} by deriving
\begin{align*}
\|\Lc^{-\alpha}(u-u_h)\|_1^2 &\leq Ca\big(\Lc^{-\alpha}(u-u_h),\Lc^{-\alpha}(u-u_h)\big)  \\
&= Ca\big(u-u_h,\Lc^{-2\alpha}(u-u_h)\big)  \\
&= Ca\big(u-u_h,(I-\Pih)\Lc^{-2\alpha}(u-u_h)\big) \\
&\leq C \|u-u_h\|_1 \|(I-\Pih)\Lc^{-2\alpha}(u-u_h)\|_1 \\
&\leq Ch^{\min(p,k+1,K+1)+\min(p,2\alpha,K+1)}\|f\|_{\min(p-1,k,K)}\| \Lc^{-\alpha}(u-u_h) \|_1,
\end{align*}
where $I$ denotes the identity operator and where we used the coercivity of $a$ in the first line, Lemma \ref{lem:L1} in the second line, the definitions of $u$ and $u_h$ in the third line, the boundedness of $c$ and the Cauchy--Schwarz inequality in the fourth line, and Lemma \ref{lem:int} and the regularity assumption in the last line.

To prove \eqref{eq:ell2b}, we first derive
\begin{align*}
\|\Lc^{-\alpha}(u-u_h)\|_0^2 &\leq C\big(\Lc^{-\alpha}(u-u_h),\Lc^{-\alpha}(u-u_h)\big)_{\rho}  \\
&= Ca\big(u-u_h,\Lc^{-(2\alpha+1)}(u-u_h)\big) ,
\end{align*}
where we used the boundedness of $\rho$ in the first line and Lemma \ref{lem:L1} and the definition of $\Lc^{-1}$ in the second line. For the rest of the proof, we can proceed in a way analogous to before.
\end{proof}

\begin{cor}
\label{cor:ell3}
Assume regularity condition \eqref{eq:reg} holds for some $K\geq 0$ and let $f\in H^k(\Omega)$, with $k\geq 0$. Furthermore, let $\mathcal{S}$ be an operator of the form
\begin{align}
\label{eq:S1}
\mathcal{S}=\Lc^{-\beta_{n+1}}(\Lc^{-1}-\Lh^{-1}\Pih)\Lc^{-\beta_{n}}(\Lc^{-1}-\Lh^{-1}\Pih) \cdots \Lc^{-\beta_2}(\Lc^{-1}-\Lh^{-1}\Pih) \Lc^{-\beta_{1}},
\end{align}
with $n\geq 1$, $\beta_i\geq 0$ for $i=1,\dots,n+1$. Define $\alpha:=n+\beta_1+\beta_2+\cdots+\beta_{n+1}$. Then
\begin{subequations}
\begin{align}
\|\mathcal{S}f\|_1 &\leq  Ch^{\min({p,2\alpha+k-1,K+1})}\|f\|_k,  \label{eq:ell3a} \\
\|\mathcal{S}f\|_0 &\leq  Ch^{\min({p,2\alpha+k,K+1})}\|f\|_k,  \label{eq:ell3b}
\end{align}
\label{eq:ell3}%
\end{subequations}
where $p\geq 1$ denotes the degree of the finite element space.
\end{cor}

\begin{proof}
The asserted follows from Lemma \ref{lem:ell2} and can be proven by induction on $n$.
\end{proof}

\begin{lem}
\label{lem:ell4}
Assume regularity condition \eqref{eq:reg} holds for some $K\geq 0$ and let $f\in H^k(\Omega)$, with $k\geq 0$. Furthermore, let $\mathcal{S}_1,\mathcal{S}_2$ be two operators of the form
\begin{align*}
\mathcal{S}=\Lc^{-\beta_{n+1}}(\Lh^{-1}\Pih)\Lc^{-\beta_{n}}(\Lh^{-1}\Pih) \cdots \Lc^{-\beta_2}(\Lh^{-1}\Pih) \Lc^{-\beta_{1}},
\end{align*}
with $n\geq 0$, $\beta_i\geq 0$ for $i=1,\dots,n+1$. Define $\alpha:=n+\beta_1+\beta_2+\cdots+\beta_{n+1}$. If $\alpha\geq 1$, then
\begin{subequations}
\begin{align}
\|(\mathcal{S}_1-\mathcal{S}_2)f\|_1 &\leq  Ch^{\min({p,2\alpha+k-1,K+1})}\|f\|_k,  \label{eq:ell4a} \\
\|(\mathcal{S}_1-\mathcal{S}_2)f\|_0 &\leq  Ch^{\min({p,2\alpha+k,K+1})}\|f\|_k,  \label{eq:ell4b}
\end{align}
\label{eq:ell4}%
\end{subequations}
where $p\geq 1$ denotes the degree of the finite element space.
\end{lem}

\begin{proof}
Replace the terms $\Lh^{-1}\Pih$ by $(\Lh^{-1}\Pih-\Lc^{-1})+\Lc^{-1}$ in both $\mathcal{S}_1$ and $\mathcal{S}_2$. The expansions of $\mathcal{S}_1$ and $\mathcal{S}_2$ then both consist of the term $\Lc^{-\alpha}$ and several terms of the form of \eqref{eq:S1}. The difference $\mathcal{S}_1-\mathcal{S}_2$ can then be written purely in terms of operators of the form of \eqref{eq:S1}. The Lemma then follows from the triangle inequality and Corollary \ref{cor:ell3}.
\end{proof}

\begin{rem}
The operators $\mathcal{S}_1$ and $\mathcal{S}_2$ are sequences of $\alpha$ operators, where each operator is either $\Lc^{-1}$ or $\Lh^{-1}\Pih$.
\end{rem}

\section{Properties of negative-order Sobolev norms}
\label{sec:nnorm}

\begin{lem}
\label{lem:nnorm}
Let $u\in L^2(\Omega)$ and assume $\rho\in\mathcal{C}^{k}$ and $c\in\mathcal{C}^{k-1}$ for some $k\geq 0$. Then
\begin{align}
\label{eq:nnorm}
\|u\|_{-k} &\leq C\|u\|_{-k*}
\end{align}
\end{lem}

\begin{proof}
Let $w\in H_0^k(\Omega)$ and suppose $k=2\alpha$. From the regularity of $\rho$ and $c$, it follows that $\Lc^\kappa (\rho^{-1}w) \in H_0^1(\Omega)$, for $\kappa\leq\alpha-1$, and $\Lc^\alpha (\rho^{-1}w) \in L^2{\Omega}$ and therefore $\Lc^{-\alpha}\Lc^{\alpha}(\rho^{-1}w)=\rho^{-1}w$. We can then derive
\begin{align}
\label{eq:nnorm1a}
(u,w) = (u,\rho^{-1}w)_{\rho} = (\Lc^{-\alpha}u,\Lc^{\alpha}\rho^{-1}w)_{\rho} \leq C\|\Lc^{-\alpha}u\|_{0} \|w\|_{2\alpha},
\end{align}
where the first equality follows from the definition of $(\cdot,\cdot)_{\rho}$, the second equality follows from $\Lc^{-\alpha}\Lc^{\alpha}(\rho^{-1}w)=\rho^{-1}w$ and Lemma \ref{lem:L1}, and the final inequality follows from the Cauchy--Schwarz inequality and the regularity of $\rho$ and $c$. In case $k=2\alpha+1$, we can, in an analogous way, derive 
\begin{align}
\label{eq:nnorm1b}
(u,w) = (u,\rho^{-1}w)_{\rho} = a(\Lc^{-(\alpha+1)}u,\Lc^{\alpha}\rho^{-1}w) \leq C\|\Lc^{-(\alpha+1)}u\|_{1} \|w\|_{2\alpha+1}.
\end{align}
Inequality \eqref{eq:nnorm} then follows immediately from \eqref{eq:nnorm1a} and \eqref{eq:nnorm1b}.
\end{proof}

\end{document}